\documentclass{amsart}%
\usepackage{amsmath}
\usepackage{amsfonts}
\usepackage{amssymb}
\usepackage{graphicx}%
\usepackage{amssymb,amsmath,amsthm,amsfonts,commath,bm,mathtools}
\newtheorem{lemma}{Lemma}[section]
\newtheorem{theorem}[lemma]{Theorem}
\newtheorem{remark}[lemma]{Remark}

\newtheorem{coro}[lemma]{Corollary}
\newtheorem{definition}[lemma]{Definition}
\newtheorem{example}[lemma]{Example}

\title[Amerio's Theorem for Remotely Almost Periodic \ldots Equations]{Amerio's Theorem for Remotely Almost Periodic Differential Equations}

\author{David Cheban}
\address[D. Cheban]{State University of Moldova\\
Institute of Mathematics and Informatics "Vladimir Andrunachievici"\\
Laboratory of Differential Equations\\
str. Academiei 5\\
MD--2028 Chi\c{s}in\u{a}u, Moldova} \email[D.
Cheban]{david.ceban@usm.md, davidcheban@yahoo.com}

\date{\today}

\begin{document}
\begin{abstract}
The aim of this paper is to study the remotely almost periodic
motions of dynamical systems and solutions of nonlinear
differential equations. We establish some properties of remotely
almost periodic motions and generalize the well known Amerio's
theorem for abstract remotely almost periodic dynamical systems.
Application of our general results for different classes of
nonlinear differential/difference and algebraic equations is
given.
\end{abstract}

\maketitle
%\newpage
%\tableofcontents
%\newpage

\section{Introduction}\label{Sec1}

Let $(\mathfrak B, |\cdot|)$ be a Banach space with the norm
$|\cdot|$. Denote by $\mathbb R =(-\infty,+\infty)$, $\mathbb
R_{+}=[0,+\infty)$, $\mathbb T\in \{\mathbb R_{+},\mathbb R\}$ and
$C(\mathbb T,\mathfrak B)$ (respectively, $C(\mathbb T\times
\mathfrak B,\mathfrak B)$) the space of all continuous functions
$\varphi :\mathbb T\to \mathfrak B$(respectively, $f:\mathbb
T\times \mathfrak B\to \mathfrak B$) equipped with the compact
open topology \cite{Kel}. Denote by $\varphi^{h}$ (respectively,
$f^{h}$ ($h\in \mathbb \mathbb T$)) the $h$-translation of
$\varphi$ (respectively,$f$) w.r.t., i.e.,
$\varphi^{h}(t):=\varphi(t+h)$ for all $t\in \mathbb T$
(respectively, $f^{h}(t,x):=f(t+h,x)$ for every $(t,x)\in \mathbb
T\times \mathfrak B$). The triplet $(C(\mathbb T,\mathfrak
B),\mathbb T,\sigma)$, where $\sigma :\mathbb T\times C(\mathbb
T,\mathfrak B)\to C(\mathbb T,\mathfrak B)$ is a mapping defined
by the equality $\sigma(h,\varphi):=\varphi^{h}$ ($(h,\varphi)\in
\mathbb T\times C(\mathbb T,\mathfrak B)$), is a shift (or
Bebutov's) dynamical system \cite[Ch.I]{Che_2015}. Denote by
$H(\varphi):=\overline{\{\varphi^{h}:\ h\in \mathbb T\}}$
(respectively, $H^{+}(f):=\overline{\{f^{h}:\ h\in \mathbb
T_{+}\}}$) the closure of the set $\{\varphi^{h}:\ h\in \mathbb
T\}$ (respectively, $\{f^{h}:\ h\in \mathbb T_{+}\}$) in the space
$C(\mathbb T,\mathfrak B)$ (respectively, in the space $C(\mathbb
T_{+}\times \mathfrak B,\mathfrak B)$), where $\mathbb
T_{+}:=\{t\in \mathbb T|\ t\ge 0\}$.

A function $\psi \in C(\mathbb T,\mathfrak B)$ (respectively,
$g\in C(\mathbb T\times \mathfrak B,\mathfrak B)$) is called
$\omega$-limit for $\varphi\in C(\mathbb T,\mathfrak B)$
(respectively, for $f\in C(\mathbb T\times \mathfrak B,\mathfrak
B)$) if there exists a sequence $h_{k}\to +\infty$
($\{h_k\}\subset \mathbb T$) such that $\varphi^{h_k}\to \psi$
(respectively, $f^{h_k}\to g$) as $k\to \infty$ in the space
$C(\mathbb T,\mathfrak B)$ (respectively, in the space $C(\mathbb
T\times \mathfrak B,\mathfrak B)$). Denote by $\omega_{\varphi}$
(respectively, $\omega_{f}$) the set of all $\omega$-limit
functions for $\varphi$ (respectively, for $f$).

Recall that a set $A\subseteq \mathbb T$ is said to be relatively
dense in $\mathbb T$ if there exists a positive number $l$ such
that
\begin{equation}\label{eqI_1}
A\bigcap [a,a+l]\not= \emptyset \nonumber
\end{equation}
for all $a\in \mathbb T$, where $[a,a+l]:=\{t\in \mathbb T:\ a\le
t\le a+l\}$.

\begin{definition}\label{defI1} A function $\varphi \in C(\mathbb T,\mathfrak
B)$ is said to be
\begin{enumerate}
\item almost periodic if for every $\varepsilon
>0$ the set
\begin{equation}\label{eqI_2}
\mathcal F(\varphi,\varepsilon):=\{\tau \in \mathbb T:\
|\varphi(t+\tau)-\varphi(t)|<\varepsilon \ \ \forall \ t\in
\mathbb T\}\nonumber
\end{equation}
is relatively dense; \item asymptotically almost periodic if there
are two functions $p,r\in C(\mathbb T,\mathfrak B)$ such that
\begin{enumerate}
\item
\begin{equation}\label{eqI_3}
\varphi(t)=p(t)+r(t)\nonumber
\end{equation}
for all $t\in \mathbb T$ and \item the function $p$ is almost
periodic and $r\in C_{0}(\mathbb T,\mathfrak B):=\{\varphi \in
C(\mathbb T,\mathfrak B):\ \lim\limits_{t\to
+\infty}|\varphi(t)|=0\}$.
\end{enumerate}
\end{enumerate}
\end{definition}

Let $Q$ be a compact subset of $\mathfrak B$.

\begin{definition}\label{defI_1}
A function $f\in C(\mathbb T\times \mathfrak B,\mathfrak B)$ is
said to be
\begin{enumerate}
\item almost periodic in $t\in \mathbb T$ uniformly w.r.t. $x\in Q$ if the set
\begin{equation}\label{eqI_4}
\mathcal T(f,Q,\varepsilon):=\{\tau \in \mathbb T:\
\max\limits_{x\in Q}|f(t+\tau,x)-f(t,x)|<\varepsilon\ \ \forall \
t\in \mathbb T\}\nonumber
\end{equation}
is relatively dense; \item asymptotically almost periodic in $t\in
\mathbb T$ uniformly w.r.t. $x\in Q$ if there are two functions
$P,R\in C(\mathbb T\times \mathfrak B,\mathfrak B)$ such that
\begin{enumerate}
\item
\begin{equation}\label{eqI_5}
f(t,x)=P(t,x)+R(t,x)\nonumber
\end{equation}
for all $(t,x)\in \mathbb T\times \mathfrak B$ and \item the
function $P$ is almost periodic in $t\in \mathbb T$ uniformly
w.r.t. $x\in Q$ and
$$
\lim\limits_{t\to +\infty}\max\limits_{x\in Q}|R(t,x)|=0 .
$$
\end{enumerate}
\end{enumerate}
\end{definition}

\begin{remark}\label{remI2} Let $Q$ be a compact subset from $\mathfrak
B$ and $C(Q,\mathfrak B)$ (respectively, $C(\mathbb T\times
Q,\mathfrak B)$) be the Banach space of all continuous functions
$\varphi :Q\to \mathfrak B$ (respectively, the space of all
continuous functions $f:\mathbb T\times Q\to \mathfrak B$)
equipped with the norm $\|\varphi\|:=\max\limits_{x\in
Q}|\varphi(x)|$ (respectively, with the compact open topology).
\begin{enumerate}
\item the spaces $C(\mathbb T\times Q,\mathfrak B)$ and $C(\mathbb
T,C(Q,\mathfrak B))$ topologically isomorphic \cite{Che_2025},
i.e., there exists a homeomorphic mapping
$$
\Phi:C(\mathbb T\times
Q,\mathfrak B)\to C(\mathbb T,C(Q,\mathfrak B))
$$
such that $\Phi(f^{h})=\Phi(f)^{h}$ for all $(f,h)\in C(\mathbb
T\times Q,\mathfrak B)\times \mathbb T$: \item the function $f\in
C(\mathbb T\times \mathfrak B,\mathfrak B)$ is almost periodic
(respectively, asymptotically almost periodic) in $t\in \mathbb T$
uniformly w.r.t. $x\in Q$ if and only if the function $F_{Q}\in
C(\mathbb T,C(Q,\mathfrak B))$ is almost periodic (respectively,
asymptotically almost periodic) \cite{Che_2025}, where
$F_{Q}(t):=f_{Q}(t,\cdot)$ for all $t\in \mathbb T$ and $f_{Q}$ is
the restriction of the function $f\in C(\mathbb T\times \mathfrak
B,\mathfrak B)$ on the set $\mathbb T\times Q$, i.e.,
$f_{Q}:=f{\big{|}_{\mathbb T\times Q}}$.
\end{enumerate}
\end{remark}

Consider a differential equation
\begin{equation}\label{eqIn1}
x'=f(t,x),
\end{equation}
where $f\in C(\mathbb R\times \mathfrak B,\mathfrak B)$. Along
with the equation (\ref{eqIn1}) we will consider the family of
equations
\begin{equation}\label{eqIn2}
y'=g(t,y),
\end{equation}
where $g\in H(f):=\overline{\{f^{h}|\ h\in \mathbb R\}}$ and by
bar the closure in the space $C(\mathbb R\times \mathfrak
B,\mathfrak B)$ is denoted.

\begin{definition}\label{defR1}
A function $f$ is called regular (respectively, positively
regular), if for every $g\in H(f)$ (respectively, $g\in H^{+}(f)$)
the equation (\ref{eqIn2}) admits a unique solution
$\varphi(t,v,g)$ passing through the point $v\in \mathfrak B$ at
the initial moment $t=0$ and defined on $\mathbb R_{+}$.
\end{definition}

\begin{definition}\label{defD1} A solution $\varphi :\mathbb T\to Q$ of
the equation (\ref{eqIn1}) is said to be separated in $Q$ on the
set $\mathbb T$ \cite{Ame1955,Fin}, if there exists a number
$d(\varphi)>0$ so that if $\psi :\mathbb T\to Q$ is an other
solution of (\ref{eqIn1}), then
\begin{equation}\label{eqD_02}
|\varphi(t)-\psi(t)|\ge d(\varphi)\nonumber
\end{equation}
for all $t\in \mathbb T$.
\end{definition}

It is well known the following Amerio's result \cite{Ame1955} (see
also \cite{AP1971}, \cite{Fin}, \cite{Lev-Zhi} and the
bibliography therein).

\begin{theorem}\label{thI1} (\it{Amerio's theorem})
Let $Q$ be a compact subset of the Banach space $\mathfrak B$.
Assume that the following conditions are fulfilled:
\begin{enumerate}
\item the equation (\ref{eqIn1}) admits a solution $\varphi
:\mathbb R \to Q$; \item the function $f\in C(\mathbb R\times
\mathfrak B,\mathfrak B)$ is almost periodic in $t\in \mathbb R$
uniformly w.r.t. $x\in Q$; \item $f$ is positively regular; \item
each equation (\ref{eqIn2}) ($g\in H(f)$) has only separated on
$\mathbb R$ solutions in $Q$.
\end{enumerate}

Then the solution $\varphi$ is almost periodic.
\end{theorem}

\begin{remark}\label{remI1} Note that Theorem \ref{thI1} remains
true in the following cases:
\begin{enumerate}
\item if we replace the condition of separation (on $\mathbb R$)
by semi-separation condition \cite{Fin_1972}, \cite[Ch.X]{Fin};
\item if the operators $f(t,\cdot):\mathfrak \to \mathfrak B$
($t\in \mathbb R$) are not obligatory continuous (see, for
example, \cite[Ch.VII,p.99]{Lev-Zhi}).
\end{enumerate}
\end{remark}

Amerio's theorem was generalized for asymptotically almost
periodic differential equations in our publication \cite{Che_1977}
(see also \cite[Ch.II,III]{Che_2009}).

\begin{theorem}\label{thI2} (\it{Amerio's theorem for asymptotically almost periodic equations})
Let $Q$ be a compact subset of the Banach space $\mathfrak B$.
Assume that the following conditions are fulfilled:
\begin{enumerate}
\item the equation (\ref{eqIn1}) admits a solution $\varphi
:\mathbb R_{+} \to Q$; \item the function $f\in C(\mathbb R\times
\mathfrak B,\mathfrak B)$ is asymptotically almost periodic in
$t\in \mathbb R$ uniformly w.r.t. $x\in Q$; \item $f$ is
positively regular; \item each equation (\ref{eqIn2}) ($g\in
\omega_{f}$) has only separated on $\mathbb R$ solutions in $Q$.
\end{enumerate}

Then the solution $\varphi$ is asymptotically almost periodic.
\end{theorem}

\begin{definition}\label{defI3} A function $\varphi \in C(\mathbb T,\mathfrak
B)$ is said to be remotely almost periodic
\cite{Che_2024.1,Che_2025.01,RS_1986} if for arbitrary positive
number $\varepsilon$ there exists a relatively dense in subset
$\mathcal P(\varepsilon,\varphi)\subseteq \mathbb T$ such that for
every $\tau \in \mathcal P(\varepsilon,\varphi)$ there exists a
number $L(\varepsilon,\varphi,\tau)>0$ for which we have
\begin{equation}\label{eqRAP1_1}
|\varphi(t+\tau)-\varphi(t))<\varepsilon \nonumber
\end{equation}
for all $t\ge L(\varepsilon,\varphi,\tau)$.
\end{definition}

\begin{definition}\label{defI5}
A function $\varphi \in C(\mathbb T,\mathfrak B)$ (respectively,
$f\in C(\mathbb T\times \mathfrak B,\mathfrak B)$) is said to be
positively Lagrange stable if the set $\{\varphi^{h}:\ h\ge 0\}$
(respectively, $\{f^{h}:\ h\ge 0\}$) is precompact in the space
$C(\mathbb T,\mathfrak B)$ (respectively, in the space $C(\mathbb
T\times \mathfrak B,\mathfrak B)$).
\end{definition}

\begin{definition}\label{defI4} Let $Q$ be a compact subset of the Banach space $\mathfrak B$.
A function $f\in C(\mathbb T\times \mathfrak B,\mathfrak B)$ is
said to be remotely almost periodic in $t\in \mathbb T$ uniformly
w.r.t. $x\in Q$ if the function $f_{Q}\in C(\mathbb
T,C(Q,\mathfrak B))$ is remotely almost periodic.
\end{definition}

\begin{remark} Note that every asymptotically almost periodic
function $\varphi \in C(\mathbb T,\mathfrak B)$ (respectively,
asymptotically almost periodic in $t\in \mathbb T$ uniformly
w.r.t. $x\in Q$ function $f\in C(\mathbb T\times Q,\mathfrak B)$)
possesses the following properties:
\begin{enumerate}
\item it is remotely almost periodic (respectively, it is remotely
almost periodic in $t\in \mathbb T$ uniformly w.r.t. $x\in Q$)
\cite{Che_2024.1,Che_2025.01}; \item the function $\varphi$
(respectively, the function $f\in C(\mathbb T\times \mathfrak
B,\mathfrak B)$) is positively Lagrange stable; \item
$\omega_{\varphi}$ (respectively $\omega_{f}$) is a minimal set
consisting of almost periodic functions (respectively, almost
periodic in $t\in \mathbb T$ uniformly w.r.t. $x$ on every compact
subset from $\mathfrak B$).
\end{enumerate}

The converse statement, generally speaking, is not true
\cite{Che_2025.01}. To prove this statement it suffices to
consider the function $\varphi \in C(\mathbb R_{+},\mathbb R)$
defined by the equality $\varphi(t)=\sin (t+\ln (1+t)))$ for all
$t\ge 0$. This function is remotely almost periodic, positively
Lagrange stable, its $\omega$-limit set is minimal, but $\varphi$
is not asymptotically almost periodic.
\end{remark}

The aim of this paper is to generalize of Amerio's theorem for
remotely almost periodic differential equations (\ref{eqIn1}).

Namely, we study the following problem.

\textbf{Problem.} Assume that the following conditions hold:
\begin{enumerate}
\item the function $f\in C(\mathbb R\times \mathfrak B,\mathfrak
B)$ is positively Lagrange stable; \item $f$ is remotely almost
periodic in $t\in \mathbb T$ uniformly w.r.t. $x$ on every compact
subset from $\mathfrak B$; \item the $\omega$-limit set
$\omega_{f}$ of the function $f$ is minimal; \item the equation
(\ref{eqIn1}) admits a precompact on $\mathbb R_{+}$ solution
$\varphi$, i.e., set $\varphi(\mathbb R_{+})$ is precompact in
$\mathfrak B$; \item each equation (\ref{eqIn2}) ($g\in
\omega_{f}$) has only separated on $\mathbb R$ solutions in $Q$;
\item the function $f$ is positively regular.
\end{enumerate}
Will there be the solution $\varphi$ remotely almost periodic?

The main result of this paper in the following theorem is
contained.

\begin{theorem}\label{thI_2} ({\it Amerio's theorem for remotely almost
periodic differential equations}) Suppose that the following
conditions hold:
\begin{enumerate}
\item the equation (\ref{eqIn1}) admits a precompact on $\mathbb
R_{+}$ solution $\varphi$; \item the function $f\in C(\mathbb
R\times \mathfrak B,\mathfrak B)$ is positively Lagrange stable
and regular; \item $f$ is remotely almost periodic in $t\in
\mathbb T$ uniformly w.r.t. $x$ on every compact subset from
$\mathfrak B$; \item the $\omega$-limit set $\omega_{f}$ of the
function $f$ is minimal; \item each equation (\ref{eqIn2}) ($g\in
\omega_{f}$) has only separated on $\mathbb R_{+}$ solutions
defined on $\mathbb R$ with the values from $Q$.
\end{enumerate}

Then the solution $\varphi$ is remotely almost periodic.
\end{theorem}

In this article, we continue the research started in the author's
works \cite{Che_2024_00}-\cite{Che_2026} in which remotely almost
periodic motions of dynamical systems and solutions of
differential equations are studied. A number of properties of
remotely almost periodic movements are established and an
important concept of remote comparability by the character of
recurrence (at the infinity) for remotely almost periodic motions
is introduced. This concept plays an important role in the study
of remotely almost periodic movements.

The notion of remotely almost periodicity (on the real axis
$\mathbb R$) for scalar function was introduced and studied by
Sarason \cite{Sar_1984}. Recall that a continuous and bounded
function $f:\mathbb R \to \mathbb R$ is said to be remotely almost
periodic if for every $\varepsilon >0$ there exists a positive
number $l$ such that on every segment $[a,a+l]\subset \mathbb R$
there exists at least one number $\tau \in [a,a+l]$ such that
\begin{equation}\label{eqI01}
    d_{\infty}(f^{\tau},f)<\varepsilon ,\nonumber
\end{equation}
where $d_{\infty}(f,g):=\limsup\limits_{|t|\to
    +\infty}|f(t)-g(t)|$ and $f^{\tau}(t):=f(t+\tau)$ for all $t\in
\mathbb R$. The remotely almost periodic ($RAP$) functions form a
closed subalgebra of $BUC$ (the algebra of bounded and uniformly
continuous complex valued functions on $\mathbb R$). The main
result of Sarason is that $RAP$ is generated as a Banach algebra,
by $AP$ (the algebra of Bohr almost periodic functions) and
another algebra $SO$ (consisting of functions which oscillate
slowly at $\infty$).

Remotely almost periodic functions on the semi-axis $\mathbb
R_{+}$ with values in the Banach space were introduced by Ruess
and Summers \cite{RS_1986} (see also Baskakov \cite{Bas_2013}).

Many interesting results by Baskakov were obtained (see, for
example, \cite{Bas_2013}-\cite{BSS_2019}). In particular, he
obtained profound results on the harmonic analysis of remotely
almost periodic functions. Finally. note that there are a number
of publications \cite{KK_2023,KK_2022,MCKP_2021} in which remotely
almost periodic solutions of some classes of differential and
difference equations are studied.

The paper is organized as follows.

In the Section \ref{Sec2}, we collect some known notions of
dynamical systems and facts about remotely almost periodic motions
and functions that we use in this paper.

The Section \ref{Sec3} is dedicated to the Amerio's theorem for
abstract almost periodic dynamical systems and  their
generalization for asymptotically almost periodic systems.

In the Section \ref{Sec4} we study the remotely almost periodic
 (abstract) nonautonomous dynamical systems and  we generalize
 the Amerio's theorem for this class of nonautonomous dynamical systems.

The Section \ref{Sec5} is dedicated to the study of asymptotically
almost periodic and remotely almost periodic  nonautonomous
dynamical systems with positively uniformly asymptotically stable
trajectories.

Finally, in the Section \ref{Sec6} we apply our general results
from Sections 3-5 for different classes of differential/difference
equations (ordinary differential equations, difference equations,
functional differential equations with finite delay and stochastic
differential equations) and algebraic equations with remotely
almost periodic coefficients.

\section{Preliminary}\label{Sec2}

Let $(X,\rho_{X})$ and $(Y,\rho_{Y})$ be two complete metric
spaces with the distance $\rho_{X}$ and $\rho_{Y}$
respectively\footnote{In what follows, in the notation $\rho_{X}$
(respectively, $\rho_{Y}$), we will omit the index $X$
(respectively, $Y$) if this does not lead to a misunderstanding.},
$\mathbb Z :=\{0,\pm 1, \pm 2, \ldots \}$, $\mathbb
N:=\{1,2,\ldots\}$, $\mathbb S =\mathbb R$ or $\mathbb Z$,
$\mathbb S_{+}=\{t \in \mathbb S |\quad t \ge 0 \}$, $\mathbb
S_{-}=\{t \in \mathbb S| \quad t \le 0 \}$, $\mathbb T \in
\{\mathbb S,\ \mathbb S_{+}\}$ and $(X,\mathbb S_{+},\pi)$
(respectively, $(Y,\mathbb S, \sigma )$) be an autonomous
one-sided (respectively, two-sided) dynamical system on $X$
(respectively, on $Y$).

Consider a dynamical system $(X,\mathbb T,\pi)$.

\begin{definition}\label{defSP1} A point $x\in X$ (respectively, a motion $\pi(t,x)$) is
said to be:
\begin{enumerate}
\item[-] stationary, if $\pi(t,x)=x$ for all $t\in \mathbb T$;
\item[-] $\tau$-periodic ($\tau >0$ and $\tau \in \mathbb T$), if
$\pi(\tau,x)=x$; \item[-] asymptotically stationary (respectively,
asymptotically $\tau$-periodic), if there exists a stationary
(respectively, $\tau$-periodic) point $p\in X$ such that
\begin{equation}\label{eqAP1*}
\lim\limits_{t\to \infty}\rho(\pi(t,x),\pi(t,p))=0.\nonumber
\end{equation}
\end{enumerate}
\end{definition}

%\begin{theorem}\label{thAAP1}\cite[Ch.I]{Che_2009} A point $x\in X$ is asymptotically $\tau$-periodic if and
%only if the sequences $\{\pi(k\tau,x)\}_{k=0}^{\infty}$ converges.
%\end{theorem}

\begin{definition}\label{defLS1} A point $\widetilde{x}\in X$ is said
to be $\omega$-limit for $x\in X$ if there exists a sequence
$\{t_k\}\subset \mathbb T$ such that $t_k\to +\infty$ and
$\pi(t_k,x)\to \widetilde{x}$ as $k\to \infty$.
\end{definition}

Denote by $\omega_{x}$ the set of all $\omega$-limit points of
$x\in X$.

\begin{definition}\label{defSAP1}  We will call
a point $x\in X$ (respectively, a motion $\pi(t,x)$) remotely
$\tau$-periodic ($\tau\in \mathbb T$ and $\tau
>0\widetilde{}$) if
\begin{equation}\label{eqSAP_1}
 \lim\limits_{t\to
\infty}\rho(\pi(t+\tau,x),\pi(t,x))=0 .
\end{equation}
\end{definition}

\begin{remark}\label{remS1.0}
The motions of dynamical systems possessing the property
(\ref{eqSAP_1}) in the works of Cryszka \cite{Gry_2018} and
Pelczar \cite{Pel_1985} were studied.
\end{remark}

\begin{definition}\label{defLS2} A point $x$ is called positively Lagrange
stable, if the semi-trajectory $\Sigma_{x}^{+}:=\{\pi(t,x)|\ t\in
\mathbb S_{+}\}$ is a precompact subset of $X$.
\end{definition}

\begin{theorem}\label{th1.3.9}\cite[Ch.I]{Che_2020} Let $x\in X$ be
positively Lagrange stable and $\tau\in\mathbb T\ (\tau >0)$. Then
the following statements are equivalent:
\begin{enumerate}
\item[a.] the motion $\pi(t,x)$ is remotely $\tau$-periodic;
\item[b.] every point $p\in\omega_{x}$ is $\tau$-periodic.
\end{enumerate}
\end{theorem}

\begin{definition}\label{defSAP2} A point $x$ (respectively, a
motion $\pi(t,x)$) is said to be remotely stationary, if it is
remotely $\tau$-periodic for every $\tau \in \mathbb T$.
\end{definition}

\begin{coro}\label{corSAP1} Let $x\in X$ be positively Lagrange stable. Then the following
statements are equivalent:
\begin{enumerate}
\item[a.] the motion $\pi(t,x)$ is remotely stationary; \item[b.]
every point $p\in\omega_{x}$ is stationary.
\end{enumerate}
\end{coro}
\begin{proof} This statement follows directly from the
corresponding definition and Theorem \ref{th1.3.9}.
\end{proof}

\begin{definition}\label{defAP1} A point $x\in X$ of the dynamical
system $(X,\mathbb T,\pi)$ is said to be:
\begin{enumerate}
\item almost recurrent if for all $\varepsilon >0$ the set
\begin{equation}\label{eqAP_01}
\mathcal T(\varepsilon,p):=\{\tau \in \mathbb T|\
\rho(\pi(\tau,p),p)<\varepsilon\} \nonumber
\end{equation}
is relatively dense in $\mathbb T$; \item recurrent, if it is
almost recurrent and Lagrange stable; \item almost periodic if for
every $\varepsilon
>0$ the set
\begin{equation}\label{eqAP_001}
\mathcal P(\varepsilon,p):=\{\tau \in \mathbb T|\
\rho(\pi(t+\tau,p),\pi(t,p))<\varepsilon \ \ \mbox{for all}\ t\in
\mathbb T\}\nonumber
\end{equation}
is relatively dense in $\mathbb T$; \item asymptotically almost
periodic if there exists an almost periodic point $p\in X$ such
that
\begin{equation}\label{eqAP3}
\lim\limits_{t\to \infty}\rho(\pi(t,x),\pi(t,p))=0.\nonumber
\end{equation}
\end{enumerate}
\end{definition}

\begin{definition}\label{defRAP1} A point $x\in X$ (respectively,
a motion $\pi(t,x)$) is said to be remotely almost periodic
\cite{RS_1986} if for arbitrary positive number $\varepsilon$
there exists a relatively dense subset $\mathcal
P(\varepsilon,x)\subseteq \mathbb T$ such that for all $\tau \in
\mathcal P(\varepsilon,x)$ there exists a number
$L(\varepsilon,x,\tau)>0$ for which we have
\begin{equation}\label{eqRAP1}
\rho(\pi(t+\tau,x),\pi(t,x))<\varepsilon \nonumber
\end{equation}
for all $t\ge L(\varepsilon,x,\tau)$.
\end{definition}

\begin{remark}\label{remAP1} Every almost periodic point $x\in X$
is remotely almost periodic.
\end{remark}

\begin{lemma}\label{lRAP_01} \cite{Che_2024.1} Every remotely $\tau$-periodic (respectively,
remotely stationary) point $x$ of the dynamical system $(X,\mathbb
T,\pi)$ is remotely almost periodic.
\end{lemma}

\begin{definition}\label{defIS1} A subset $M\subseteq X$ is said
to be positively invariant (respectively, negatively invariant or
invariant) if $\pi(t,M)\subseteq M$ (respectively, $M\subseteq
\pi(t,M)$ or $\pi(t,M)=M$) for all $t\in \mathbb T$.
\end{definition}

\begin{theorem}\label{thLS1}\cite{Che_2020}, \cite{sib} Assume that the point $x\in X$
is positively Lagrange stable, then the following statement hold:
\begin{enumerate}
\item $\omega_{x}\not= \emptyset$; \item $\omega_{x}$ is a compact
subset of $X$; \item the set $\omega_{x}$ is invariant, that is,
$\pi(t,\omega_{x})=\omega_{x}$ for all $t\in \mathbb T$.
\end{enumerate}
\end{theorem}

\begin{definition}\label{defIC1} A closed invariant subset A of X is said to be invariantly
connected \cite{Rau_2002} if it cannot be represented as the union
of two nonempty, disjoint, closed, invariant sets.
\end{definition}

\begin{lemma}\label{lOLS1} \cite{AE_2020, bro84,
Rau_2002} If the point $x$ is positively Lagrange stable, then the
omega-limit set $\omega_{x}$ of the point $x$ is invariantly
connected.
\end{lemma}

\begin{definition}\label{defRAP2} A subset $M$ is said to be
equi-almost periodic if for every $\varepsilon >0$ there exists a
relatively dense in $\mathbb T$ subset $\mathcal P(\varepsilon,M)$
such that
\begin{equation}\label{eqRAP2}
\rho(\pi(t+\tau,p),\pi(t,p))<\varepsilon \nonumber
\end{equation}
for all $t\in \mathbb T$, $\tau \in \mathcal P(\varepsilon,M)$ and
$p\in M$.
\end{definition}

\begin{definition}\label{def1.0.4}
A closed positively invariant set, which does not contain own
closed positively invariant subset, is called \emph{minimal}.
\end{definition}

It easy to see that every positively invariant minimal set is
invariant.

\begin{lemma}\label{lSM1} \cite{Che_2024.1} Let $M$ be a compact
invariant minimal set of the dynamical system $(X,\mathbb T,\pi)$
consisting of almost periodic motions. Then the set $M$ is
equi-almost periodic.
\end{lemma}

\begin{coro}\label{corSM1} \cite{Che_2024.1} Let $X$ be a complete metric space, $(X,\mathbb
T,\pi)$ be a dynamical system and $p\in X$ be an almost periodic
point. Then the set $M:=H(p)=\overline{\{\pi(t,p)|\ t\in \mathbb
T\}}$ is equi-almost periodic.
\end{coro}

\begin{lemma}\label{lRAP1} \cite{Che_2024.1}, \cite{RS_1986} Let $x\in X$ be a
positively Lagrange stable point of the dynamical system
$(X,\mathbb T,\pi)$. The motion $\pi(t,x)$ is remotely almost
periodic if and only if its $\omega$-limit set $\omega_{x}$ is
equi-almost periodic.
\end{lemma}

\begin{theorem}\label{thAP4} \cite{Che_2024.1}
Let $(X,\mathbb T,\pi)$ be a semigroup dynamical system and $x\in
X$ be a positively Lagrange stable point. Then the following
statements are equivalent:
\begin{enumerate}
\item for arbitrary $\varepsilon >0$ there exists a relatively
dense in $\mathbb T$ subset $\mathcal P(\varepsilon,x)$ with the
property that for all $\tau \in \mathcal P(\varepsilon,x)$ there
exists a positive number $L(\varepsilon,x,\tau)$ such that
\begin{equation}\label{eqA10}
\rho(\pi(t+\tau,x),\pi(t,x))<\varepsilon
\end{equation}
for all $t\ge L(\varepsilon,x,\tau)$; \item for arbitrary
$\varepsilon >0$ there exists a relatively dense in $\mathbb S$
subset $\mathcal P(\varepsilon,x)$ with the property that for
every $\tau \in \mathcal P(\varepsilon,x)$ there exists a positive
number $L(\varepsilon,x,\tau)$ such that (\ref{eqA10}) holds for
every $t\ge L(\varepsilon,x,\tau)$ and $t+\tau \ge
L(\varepsilon,x,\tau)$.
\end{enumerate}
\end{theorem}

\begin{lemma}\label{lRAP_010} \cite{Che_2024.1} A point $x$
is remotely $\tau$-periodic (respectively, remotely stationary) if
and only if for every $\varepsilon
>0$ there exists a relatively dense in $\mathbb T$ subset $\mathcal
P(x,\varepsilon)$ such that $\{\tau \mathbb Z\}\bigcap \mathbb T
\subset \mathcal P(x,\varepsilon)$ (respectively, $\mathbb T
\subseteq \mathcal P(x,\varepsilon)$) and for every $\tau \in
\mathcal P(x,\varepsilon)$ there exists a number
$L(x,\varepsilon,\tau)>0$ for which we have
\begin{equation}\label{eqE01}
\rho(\pi(t+\tau,x),\pi(t,x))<\varepsilon \nonumber
\end{equation}
for all $t\ge L(x,\varepsilon,\tau)$.
\end{lemma}

Let $(Y ,\mathbb T, \sigma )$ be a dynamical system on $Y $ and
$E$ be a metric space.

\begin{definition}\label{def2.1}\rm %(Cocycle on the state space $E$ with the base $(Y ,\mathbb R ,\sigma)$).
A triplet $\langle E, \varphi ,(Y , \mathbb T, \sigma )\rangle$
(or briefly $\varphi$ if no confusion) is said to be a {\em
cocycle} on state space (or fibre) $E$ with the base $(Y,\mathbb
T,\sigma )$ if the mapping $\varphi : \mathbb T _{+} \times  E
\times Y \to E $ satisfies the following conditions:
\begin{enumerate}
\item $\varphi (0,u,y)=u $ for all $u\in E$ and $y\in Y$; \item
$\varphi (t+\tau,u,y)=\varphi (t, \varphi (\tau,u, y),
\sigma(\tau,y))$ for all $ t, \tau \in \mathbb T_{+},u \in E$ and
$ y \in Y$; \item the mapping $\varphi $ is continuous.
\end{enumerate}
\end{definition}

\begin{definition}\label{def2.2}
Let $ \langle E , \varphi,(Y,\mathbb T,\sigma)\rangle $ be a
cocycle on $E$, $X:=E\times Y$ and $\pi $ be a mapping from $
\mathbb T _{+} \times X $ to $X$ defined by $\pi :=(\varphi
,\sigma)$, i.e., $\pi (t,(u,y))=(\varphi (t,u,y),\sigma(t,y))$ for
all $ t\in \mathbb T_{+}$ and $(u,y)\in E\times Y$. The triplet $
(X,\mathbb T_{+}, \pi)$ is an autonomous dynamical system and
called a skew-product dynamical system.
\end{definition}

\begin{definition}\label{def2.3} \rm %(Non-autonomous dynamical system.)
Let $\mathbb T_{1}\subseteq \mathbb T _{2} $ be two sub-semigroups
of the group $\mathbb S$, $(X,\mathbb T _{1},\pi )$ and $(Y
,\mathbb T _{2}, \sigma )$ be two autonomous dynamical systems and
$h: X \to Y$ be a homomorphism from $(X,\mathbb T _{1},\pi )$ to
$(Y , \mathbb T _{2},\sigma)$ (i.e., $h(\pi(t,x))=\sigma(t,h(x)) $
for all $t\in \mathbb T _{1} $ and $ x \in X $, and $h $ is
continuous and surjective), then the triplet $\langle (X,\mathbb T
_{1},\pi ),$ $ (Y ,$ $\mathbb T _{2},$ $ \sigma ),h \rangle $ is
called a non-autonomous dynamical system (shortly NDS).
\end{definition}

\begin{example} \label{ex2.4} \rm %(The non-autonomous dynamical system generated by cocycle $\varphi$.)
An important class of NDS are generated from cocycles. Indeed, let
$\langle E, \varphi ,(Y ,\mathbb T,\sigma)\rangle $ be a cocycle,
$(X,\mathbb T_{+},\pi ) $ be the associated skew-product dynamical
system ($X=E\times Y, \pi =(\varphi,\sigma)$) and $h= pr _{2}: X
\to Y$ (the natural projection mapping), then the triplet $\langle
(X,\mathbb T_{+},\pi ),$ $(Y ,\mathbb T,\sigma ),h \rangle $ is an
NDS.
\end{example}

\section{Amerio's Theorem for Nonautonomous Dynamical Systems and
Their Generalizations}\label{Sec3}

Let
\begin{equation}\label{eqNDSA1}
\langle (X,\mathbb T_{1},\pi)(Y,\mathbb T_{2},\sigma),h\rangle
\end{equation}
be a nonautonomous dynamical system (NDS) and $M$ be a nonempty
subset of $X$.

Denote by $X_{y}:=\{x\in X|\ h(x)=y\}$ and $M_{y}:= M \bigcap
X_{y}$.

\begin{definition}
A point $x\in M$ (respectively, a motion $\pi(t,x)$) is called
\cite[Ch.II]{Che_2009} separated in the set $M$ w.r.t. NDS
(\ref{eqNDSA1}) if $M_{y}$ ($y:=h(x)$) consists of a single point
$x$, i.e., $M_{y}=\{x\}$ or there exists a number $r>0$ such that
whatever would be a point $p\in M_{y}$ ($p\neq x$), $\rho(x t,p
t)\geq r$ for all $t\in\mathbb{T}_{1}$.
\end{definition}

\begin{lemma}\label{l2.3.3} \cite[Ch.II]{Che_2009}
Let $M\subseteq X$ be a compact set. If every point $x\in M$ is
separated in $M$ w.r.t. NDS (\ref{eqNDSA1}), then the set $M_{y}$
($y=h(x)$) consists of a finite number of points from $M$.
\end{lemma}

\begin{lemma}\label{l2.3.4} \cite[Ch.II]{Che_2009}
Let $Y$ be a compact and minimal set of $(Y,\mathbb
T_{2},\sigma)$, $M\subseteq X$ be a compact and positively
invariant set. If the points (motions) from $M$ of NDS
(\ref{eqNDSA1}) are separated in the set $M$, then there exists a
number $r>0$ such that
\begin{equation}\label{eqNDS_2}
\rho(p_1 t,p_2 t)\geq r>0\nonumber
\end{equation}
for all $t\in\mathbb T$ and $p_1,p_2\in M$ with $h(p_1)=h(p_2)$
and $p_1\not= p_2$.
\end{lemma}

Below we formulate a well-known theorem of Amerio \cite{Ame1955},
\cite{AP1971} for abstract nonautonomous dynamical systems (for
more details see, for example, \cite[Ch.IV]{Bro79},
\cite[Ch.VII]{Lev-Zhi}, \cite{SS77}, \cite[Ch.I]{Zhi69}).

\begin{theorem}\label{thNDS1} (\it{Amerio's theorem for NDS})
Assume that the following conditions hold:
\begin{enumerate}
\item $M\subseteq X$ is a compact and positively invariant subset
of the dynamical system $(X,\mathbb T_{1},\pi)$; \item $Y$ is a
compact and minimal set of the dynamical system $(Y,\mathbb
T_{2},\sigma)$; \item the points (motions) from $M$ of NDS
(\ref{eqNDSA1}) are separated in the set $M$.
\end{enumerate}

Then the following statements are fulfilled:
\begin{enumerate}
\item the set $M_{y}$ ($y=h(x)$) consists of a finite number
$n_{y}$ of points from $M$; \item the number $n_{y}$ is
independent of the point $y\in Y$, i.e., there exist a number
$m\in \mathbb N$ such that $n_{y}=m$ for all $y\in Y$ and
$M_{y}:=\{p_1(y),\ldots,p_m(y)\}$; \item there exists a number
$r>0$ such that
\begin{equation}\label{eqNDS_02}
\rho(\pi(t,p_{i}(y)),\pi(t,p_{j}(y)))\geq r>0\nonumber
\end{equation}
for all $t\in\mathbb T_{1}$ and $i,j=1,\ldots,m$ with $i\not= j$;
\item the set $M$ consists of recurrent points (motions); \item if
the set $Y$ consists of almost periodic points (motions), then the
set $M$ also consists of almost periodic points (motions); \item
there are a finite number of minimal sets $M^{1},\ldots,M^{k}$
($k\le m$) of the dynamical system $(X,\mathbb T_{1},\pi)$ such
that:
\begin{enumerate}
\item $M^{i}\bigcap M^{j}=\emptyset$ for every $i,j=1,\ldots,k$
with $i\not= j$; \item
\begin{equation}\label{eqNDS_3}
M=M^{1}\bigcup\ldots \bigcup M^{k} ;\nonumber
\end{equation}
\item for every $i\in \{1,\ldots,k\}$ and $y\in Y$ the set
$M_{y}^{i}:=M^{i}\bigcap X_{y}$ consists of a finite number
$m^{i}_{y}$ of points and the number $m^{i}_{y}$ does not depend
of the point $y\in Y$, i.e., for every $i\in \{1,\ldots,k\}$ there
exists a natural number $m^{i}\in \mathbb N$ such that
$m^{i}_{y}=m^{i}$ for all $y\in Y$; \item $m^{1}+\ldots +m^{k}=m$;
\item if the minimal set $Y$ consists of almost periodic points
(motions), then the sets $M^{i}$ ($i=1,\ldots,k$) consists of
almost periodic points (motions).
\end{enumerate}
\end{enumerate}
\end{theorem}

\begin{definition}\label{defAR1} A point $y\in Y$ is said to be
asymptotically recurrent if there exists a recurrent point $q\in
Y$ such that
\begin{equation}\label{eqAR1}
\lim\limits_{t\to \infty}\rho(\sigma(t,y),\sigma(t,q))=0.\nonumber
\end{equation}
\end{definition}

\begin{lemma}\label{l2.3.5}\cite[Ch.II]{Che_2009} Assume that the
following conditions are fulfilled:
\begin{enumerate}
\item the point $y\in Y$ is asymptotically recurrent
(respectively, asymptotically almost periodic); \item the point
$x\in X_{y}$ is positively Lagrange stable; \item the points
(motions) from $\omega_{x}$ of NDS (\ref{eqNDSA1}) are separated
in the set $\omega_{x}$.
\end{enumerate}

Then there exists a unique recurrent (respectively, almost
periodic) point $p\in\omega_x$ such that
\begin{equation}\label{eqNDS_4}
\lim\limits_{t\to +\infty}\rho(\pi(t,x),\pi(t,p))=0 .\nonumber
\end{equation}
\end{lemma}

\begin{theorem}\label{t2.3.6} \cite[Ch.II]{Che_2009}
Let a point $y\in Y$ be asymptotically almost periodic
$($respectively, asymptotically recurrent$)$ and let a point $x\in
X_{y}$ be positively Lagrange stable. If the points (motions) from
$\omega_{x}$ of NDS (\ref{eqNDSA1}) are separated in the set
$\omega_{x}$, then the point $x$ is asymptotically almost periodic
$($respectively, asymptotically recurrent$)$.
\end{theorem}

\section{Amerio's theorem for remotely almost periodic
NDS}\label{Sec4}

In this section we generalize Amerio's theorem for remotely almost
periodic nonautonomous dynamical systems.

\begin{theorem}\label{thNDS2}
Assume that the following conditions hold:
\begin{enumerate}
\item the point $y\in Y$ is positively Lagrange stable and
remotely almost periodic; \item the set $\omega_{y}$ is minimal;
\item the point $x\in X_{y}$ is positively Lagrange stable; \item
the points (motions) from $\omega_{x}$ of NDS (\ref{eqNDSA1}) are
separated in the set $\omega_{x}$.
\end{enumerate}

Then the following statements are fulfilled:
\begin{enumerate}
\item for every $q\in \omega_{y}$ the set
$\omega^{q}_{x}:=\omega_{x}\bigcap X_{q}$ consists of a finite
number $m_{q}$ of points from $\omega_{x}$; \item the number
$m_{q}$ is independent of the point $q\in \omega_{y}$, i.e., there
exist a natural number $m\in \mathbb N$ such that $m_{q}=m$ for
every $q\in \omega_{y}$ and
$\omega_{x}^{q}=\{p_(q),\ldots,p_m(q)\}$; \item there exists a
number $r>0$ such that
\begin{equation}\label{eqNDS2}
\rho(\pi(t,p_{i}(q)),\pi(t,p_{j}(q)))\geq r>0 \nonumber
\end{equation}
for all $q\in \omega_{y}$, $t\in\mathbb T_{1}$,
$p_{i}(q),p_{j}(q)\in \omega_{x}^{q}$ and $i,j=1,\ldots,m$ with
$i\not= j$; \item the set $\omega_{x}$ consists of almost periodic
points (motions); \item $\omega_{x}$ is a minimal set of the
dynamical system $(X,\mathbb T_{1},\pi)$; \item the point $x$
(respectively, the motion $\pi(t,x)$) is remotely almost periodic.
\end{enumerate}
\end{theorem}
\begin{proof}
Let $y\in Y$ be a positively Lagrange stable and remotely almost
periodic point and $x\in X_{y}$ be a positively Lagrange stable
point. Then the $\omega$-limit set $\omega_{x}$ is nonempty,
compact and invariant. From the above it follows that the NDS
(\ref{eqNDSA1}) induces the (two-sided) NDS $\langle
(\omega_{x},\mathbb S,\pi),(\omega_{y},\mathbb
S,\sigma),h\rangle$. Since the points of $\omega_{x}$ are
separated in the set $\omega_{x}$ and the set $\omega_{y}$ is
minimal then by Theorem \ref{thNDS1} the statements (i)-(iv) hold
and there are finite number of minimal sets $M_{1},\ldots,M_{k}$
such that
\begin{equation}\label{eqOM1}
\omega_{x}=M^{1}\bigcup \ldots \bigcup M^{k} .
\end{equation}
By Lemma \ref{lOLS1} the set $\omega_{x}$ is invariantly connected
and, consequently, in the relation (\ref{eqOM1}) the number $k$ is
equal to $1$. This means that the set $\omega_{x}$ is minimal.

Now, to finish the proof of Theorem we will establish that the
point $x$ (respectively, the motion $\pi(t,x)$) is remotely almost
periodic. To this end we note that the point $x$ (respectively,
the motion $\pi(t,x)$) is positively Lagrange stable and its
$\omega$-limit set $\omega_{x}$ is a compact and minimal set
consisting of almost periodic points. By Lemma \ref{lSM1} the set
$\omega_{x}$ is equi-almost periodic and according to Lemma
\ref{lRAP1} the point $x$ (respectively, the motion $\pi(t,x)$) is
remotely almost periodic. Theorem is completely proved.
\end{proof}

\textbf{Open problem.} The question, is whether Theorem
\ref{thNDS2} remains true in the general case when the set
$\omega_{y}$ is not minimal, is open.

\begin{coro}\label{corOM1} Let the points $y$ and $x\in X_{y}$
be positively Lagrange stable. Under the conditions of Theorem
\ref{thNDS2} the following statements hold:
\begin{enumerate}
\item if the point $y$ is remotely $\tau$-periodic, then the point
$x$ is remotely $\tau'$-periodic, where $\tau'=m\tau$ and $m$ is
some natural number; \item if the point $y$ is remotely
stationary, then
\begin{enumerate}
\item the point $x$ is asymptotically stationary if the time
$\mathbb T_{1}$ is continuous (i.e., $\mathbb R_{+}\subseteq
\mathbb T_{1}\subseteq \mathbb R$); \item the point $x$ is
asymptotically periodic, if the time $\mathbb T_{1}$ is discrete
(i.e., $\mathbb Z_{+}\subseteq \mathbb T_{1}\subseteq \mathbb Z$).
\end{enumerate}
\end{enumerate}
\end{coro}
\begin{proof}
According to Theorem \ref{thNDS2} the point $x$ is remotely almost
periodic and $\omega_{x}$ is a minimal set consisting of almost
periodic points.

Let the point $y$ be remotely $\tau$-periodic. Since the set
$\omega_{y}$ is minimal then
\begin{equation}\label{eqOM2}
\omega_{y}=\{\sigma(t,q):\ 0\le t< \tau\},
\end{equation}
where $q$ is some ($\tau$-periodic) point from $\omega_{y}$. By
Theorem \ref{thNDS2} there exists a natural number $m$ such that
for every $q\in \omega_{y}$ there are points $p_1
=p_1(q),\ldots,p_{m}=p_{m}(q)$ such that
\begin{equation}\label{eqOM3}
\omega_{x}^{q}=\{p_1,\ldots,p_{m}\} .
\end{equation}
Since the set $\omega_{x}$ is minimal then from (\ref{eqOM2}) and
(\ref{eqOM3}) it follows that every point $p\in \omega_{x}$ is
$\tau'=m\tau$-periodic. By Theorem \ref{th1.3.9} the point $x$ is
remotely $\tau'$-periodic.

Assume that the point $y$ is remotely stationary. Since the set
$\omega_{y}$ is minimal then the set $\omega_{y}$ consists of a
singe (stationary) point, i.e., there exists a stationary point
$q\in Y$ such that $\omega_{y}=\{q\}$. Since the point $y$ is
positively Lagrange stable then
\begin{equation}\label{eqOM4}
\lim\limits_{t\to \infty}\rho(\sigma(t,y),q)=0 .
\end{equation}
From the equality (\ref{eqOM4}) follows that $\omega_{x}\subseteq
X_{q}$ and by (\ref{eqOM3}) we obtain
\begin{equation}\label{eqOM5}
\omega_{x}=\{p_1,\ldots,p_{m}\} .
\end{equation}

1. If the time $\mathbb T_{1}$ is continuous, then the
$\omega$-limits set $\omega_{x}$ is connected and, consequently,
from the equality (\ref{eqOM5}) follows that the set $\omega_{x}$
consists of a single point $\{p\}$. From this fact we obtain
\begin{equation}\label{eqOM6}
\lim\limits_{t\to \infty}\rho(\pi(t,x),p)=0,\nonumber
\end{equation}
i.e., the point $x$ is asymptotically stationary.

2. If the time $\mathbb T_{1}$ is discrete, then from the equality
(\ref{eqOM5}) and the minimality of the set $\omega_{x}$ we obtain
that the set $\omega_{x}$ coincides with trajectory of the
$m$-periodic point $\bar{p}$, i.e.,
$\omega_{x}=\{\bar{p},\pi(1,\bar{p}),\ldots,\pi(m-1,\bar{p})\}$.
By Lemma \ref{l2.3.5} there exists a unique ($m$-periodic) point
$p=\pi(k_0,\bar{p})$ such that
\begin{equation}\label{eqOM7}
\lim\limits_{t\to \infty}\rho(\pi(t,x),\pi(t,p))=0 ,\nonumber
\end{equation}
i.e., the point $x$ is asymptotically $m$-periodic. Corollary is
proved.
\end{proof}

\section{Remotely Almost Periodic NDSs with Positively Uniformly
Asymptotically Stable Trajectories}\label{Sec5}

Let $\langle (X,\mathbb S_{+},\pi), (Y,\mathbb S,\sigma),h\rangle$
be a nonautonomous dynamical system.

\begin{definition}\label{defPUS1} A subset $A\subseteq X$ said to be positively
uniformly stable \cite[ChIV]{Bro79} if for arbitrary $\varepsilon
>0$ there exists $\delta =\delta(\varepsilon)>0$ such that
$\rho(x,a)<\delta$ ($a\in A,\ x\in X$ and $h(a)=h(x)$) implies
$\rho(\pi(t,x),\pi(t,a))<\varepsilon$ for all $t\ge 0.$
\end{definition}

\begin{remark} Let $A\subseteq X$ be a positively uniformly stable
subset and $B\subseteq A$, then $B$ is also positively uniformly
stable.
\end{remark}

\begin{lemma}\label{lBC1}\cite[ChIV]{Bro79}, \cite{BCh_1974}
If the set $A\subseteq X$ is positively uniformly stable and the
mapping $h:X\to Y$ is open, then the set $\overline{A}$ is so,
where $\overline{A}$ is the closure of the set $A$.
\end{lemma}

\begin{definition}\label{defPUS_2}
A point $x_0\in X$ is called \emph{positively uniformly
stable}\index{positively uniformly stable point} if the set
$\Sigma_{x_0}^{+}:=\{\pi(t,x_0):\ t\ge 0\}$ is so.
\end{definition}

\begin{coro}\label{lUS1}\cite[ChIV]{Bro79}
If $x_0$ is positively uniformly stable and $h$ is open, then:
\begin{enumerate}
\item
 $H^{+}(x_0):=\overline{\Sigma}^{+}_{x_0}$ is positively uniformly stable,
 where by bar the closure in $X$ is denoted;
\item
 $\omega_{x_0}$ is positively uniformly stable, because $\omega_{x_0}\subseteq H^{+}(x_0)$.
\end{enumerate}
\end{coro}

\begin{theorem}\label{th02}\cite[ChIV]{Che_2024} Let $\langle (X,\mathbb S_{+},\pi ),
(Y,\mathbb S,\sigma ), h\rangle $ be a nonautonomous dynamical
system with the following properties:
\begin{enumerate}
\item there exists a point $x_0\in X$ such that the positive
semi-trajectory $\Sigma_{x_0}^{+}$ is precompact; \item the set
$\omega_{x_0}$ is positively uniformly stable.
\end{enumerate}
Then the following statements hold:
\begin{enumerate}
\item[--] every motions on $\omega_{x_0}$ may be extended uniquely
to the left and on $\omega_{x_0}$ a two-sided dynamical system
$(\omega_{x_0},\mathbb S, \pi)$ is defined, i.e., the one-sided
dynamical system $(X,\mathbb S_+,\pi)$ generates on $\omega_{x_0}$
a two-sided dynamical system $(\omega_{x_0},\mathbb S,\pi )$;
\item[--] the non-autonomous dynamical system
$\langle(\omega_{x_0},$ $\mathbb S,$ $\pi ),$ $(\omega_{y_0},$
$\mathbb S,$ $\sigma),$ $h\rangle$ is negatively distal.
\end{enumerate}
\end{theorem}

\begin{remark}\label{remC2.7} Note that Theorem \ref{th02} is well
known (see, for example, \cite{NOS_2007} and references therein)
if $Y$ is a compact minimal set. In our case $Y$, generally
speaking, is not compact and it is not minimal.
\end{remark}

Let $(Y,\mathbb S,\sigma)$ be a two-sided dynamical system and
$y\in Y$. Denote by $\mathfrak N_{y}:=\{\{t_n\}\subset \mathbb S|\
$ such that the sequence $\{\sigma(t_n,y)\}$ converges$\}$ and by
$\mathfrak N_{y}^{\pm \infty}:=\{\{t_n\}\subset \mathbb S|$ such
that the sequence $\{\sigma(t_n,y)\}$ converges and $t_n\to \pm
\infty$ as $n\to \infty\}$. Let $X$ be a compact metric space,
denote by $X^{X}$ the Cartesian product of $X$ copies of the space
$X$ equipped with the Tychonov's topology. If $(X,\mathbb T,\pi)$
is a dynamical system on $X$, then for every $t\in \mathbb T$ we
denote by $\pi^{t}:=\pi(t,\cdot):X\to X$ ($x\to
\pi^{t}(x)=\pi(t,x)$) and
\begin{equation}\label{eqES1}
E(X,\mathbb T,\pi):=\overline{\{\pi^{t}|\ t\in \mathbb
T\}},\nonumber
\end{equation}
where by bar the closure of the set $\{\pi^{t}|\ t\in \mathbb T\}$
in the space $X^{X}$ is denoted.

Consider a nonautonomous dynamical system $\langle (X,\mathbb
T_{1},\pi),(Y,\mathbb T_{2},\sigma),h\rangle$. Assume that spaces
$X$ and $Y$ are compact and $y\in Y$ be a positively
(respectively, negatively) Poisson stable point. We denote by
$\mathcal E_{y}^{\pm}:=\{\xi \in E(X,\mathbb T_{1},\pi)|$ there
exists a sequence $t_n\to \pm \infty$ such that $\pi^{t_n}\to \xi$
as $n\to \infty\}$.

If the point $y\in Y$ is Poisson stable (that is, the point $y$ is
positively and negatively Poisson stable), then we denote by
\begin{equation}\label{eqES3}
\mathcal E_{y}:=\{\xi \in E(X,\mathbb T_{1},\pi)|\
\xi(X_{y})\subseteq X_{y}\},\nonumber
\end{equation}
where $X_{y}:=h^{-1}(y)=\{x\in X|\ h(x)=y\}$.

\begin{lemma}\label{lP1*}\cite[ChIII]{Che_2020} Suppose that the following conditions
are fulfilled:
\begin{enumerate}
\item $y\in Y$  is a two-sided Poisson stable point; \item
$\langle (X,\mathbb T,\pi),$ $(Y,\mathbb T,\sigma),h\rangle $ is a
two-sided nonautonomous dynamical system, i.e., $\mathbb T=\mathbb
S$; \item $X$ is a compact space; \item
\begin{equation}\label{eqP1**}
\inf \limits _{t\le 0}\rho (\pi(t,x_{1}),\pi(t,x_{2}))>0\nonumber
\end{equation}
for every $x_{1},x_{2} \in X_{y}\quad (x_{1}\not= x_{2})$.
\end{enumerate}

Then the following statement hold:
\begin{enumerate}
\item $\mathcal{E}^{-}_{y}$ and $\mathcal{E}^{+}_{y}$ are
subgroups of the semigroup $ \mathcal{E}_{y}$; \item for every
pair of points $x_1,x_2\in X_{y}$ with $x_1\not= x_2$ there are
the sequences $\{t^{-}_{n}\}\in\mathfrak N^{-\infty}_{y}$ and
$\{t^{+}_{n}\}\in\mathfrak N^{+\infty}_{y}$ such that
\begin{equation}\label{eqP_1.1}
\lim\limits_{n\to \infty}\pi(t_{n}^{\pm},x_{i})=x_{i}\ (i=1,2)
.\nonumber
\end{equation}
\end{enumerate}
\end{lemma}

\begin{lemma}\label{lP2*}\cite[ChIII]{Che_2020} Suppose that the following conditions
are fulfilled:
\begin{enumerate}
\item $y\in Y$  is a two-sided Poisson stable point; \item
$\langle (X,\mathbb T,\pi),$ $(Y,\mathbb T,\sigma),h\rangle $ is a
two-sided nonautonomous dynamical system; \item $X$ is a compact
space; \item
\begin{equation}\label{eqP1*}
\inf \limits _{t\ge 0}\rho (x_{1}t,x_{2}t)>0\nonumber
\end{equation}
for every $x_{1},x_{2} \in X_{y}\quad (x_{1}\not= x_{2})$.
\end{enumerate}

Then the following statement hold:
\begin{enumerate}
\item $\mathcal{E}^{-}_{y}$ and $\mathcal{E}^{+}_{y}$ are
subgroups of the semigroup $ \mathcal{E}_{y}$; \item for every
pair of points $x_1,x_2\in X_{y}$ with $x_1\not= x_2$ there are
the sequences $\{t^{-}_{n}\}\in\mathfrak N^{-\infty}_{y}$ and
$\{t^{+}_{n}\}\in\mathfrak N^{+\infty}_{y}$ such that
\begin{equation}\label{eqP_1.1*}
\lim\limits_{n\to \infty}\pi(t_{n}^{\pm},x_{i})=x_{i}\ (i=1,2)
.\nonumber
\end{equation}
\end{enumerate}
\end{lemma}

\begin{theorem}\label{th02.1} Let $\langle (X,\mathbb S_{+},\pi ),
(Y,\mathbb S,\sigma ), h\rangle $ be a nonautonomous dynamical
system with the following properties:
\begin{enumerate}
\item there exists a point $x_0\in X$ such that the positive
semi-trajectory $\Sigma_{x_0}^{+}$ is precompact; \item the set
$\omega_{x_0}$ is positively uniformly stable; \item every point
$q\in \omega_{y_{0}}$ ($y_0:=h(x_0)$) is two-sided Poisson stable.
\end{enumerate}
Then the following statements hold:
\begin{enumerate}
\item[--] all motions on $\omega_{x_0}$ may be extended uniquely
to the left and on $\omega_{x_0}$ is defined a two-sided dynamical
system $(\omega_{x_0},\mathbb S, \pi)$, i.e., the one-sided
dynamical system $(X,\mathbb S_+,\pi)$ generates on $\omega_{x_0}$
a two-sided dynamical system $(\omega_{x_0},\mathbb S,\pi )$;
\item[--] the non-autonomous dynamical system
$\langle(\omega_{x_0},$ $\mathbb S,$ $\pi ),$ $(\omega_{y_0},$
$\mathbb S,$ $\sigma),$ $h\rangle$ is two-sided distal.
\end{enumerate}
\end{theorem}
\begin{proof}
By Theorem \ref{th02} to prove this statement it suffices to
show that the nonautonomous dynamical system
$\langle(\omega_{x_0},$ $\mathbb S,$ $\pi ),$ $(\omega_{y_0},$
$\mathbb S,$ $\sigma),$ $h\rangle$ is positively distal, i.e.,
\begin{equation}\label{eqPD_1}
\inf\limits_{t\ge 0}\rho(\pi(t,x_1^{0}),\pi(t,x_2^{0}))>0
\end{equation}
for all $x_x,x_2\in \omega_{x_0}$ with $x_1\not= x_2$ and
$h(x_1)=h(x_2)$.

If we assume that the relation (\ref{eqPD_1}) is not true, then
there are $q\in \omega_{y_0}$, $x_1^{0},x_2^{0}\in
\omega_{y_0}^{q}:=\omega_{y_0}\bigcap X_{q}$ with $x_1^{0}\not=
x_2^{0}$ and the sequence $t_n\to +\infty$ as $n\to \infty$ such
that
\begin{equation}\label{eqPD_2}
\lim\limits_{n\to
+\infty}\rho(\pi(t_{n},x_1^{0}),\pi(t_{n},x_2^{0}))=0
\end{equation}
for all $n\in \mathbb N$.

Since the set $\omega_{x_0}$ is uniformly positively stable then
from the relation (\ref{eqPD_2}) we obtain
\begin{equation}\label{eqPD4}
\lim\limits_{t\to \infty}\rho(\pi(t,x_{1}^{0}),\pi(t,x_{2}^{0}))=0
.
\end{equation}

On the other hand by Lemma \ref{lP2*} there exists a sequence
$t_{n}^{+}\to +\infty$ as $n\to \infty$ such that
\begin{equation}\label{eqPD3}
\lim\limits_{n\to \infty}\pi(t_{n}^{+},x_{i}^{0})=x_{i}^{0}\
(i=1,2).
\end{equation}
Note that
\begin{equation}\label{eqPD3.1}
0<\rho(x_1^{0},x_2^{0})\le
\rho(x_1^{0},\pi(t_{n}^{+},x_{1}^{0}))+\rho(\pi(t_{n}^{+},x_{1}^{0}),\pi(t_{n}^{+},x_{2}^{0}))+\rho(\pi(t_{n}^{+},x_{2}^{0}),x_{2}^{0})
\end{equation}
for all $n\in \mathbb N$. Passing to the limit in (\ref{eqPD3.1})
as $n\to \infty$ and taking into account
(\ref{eqPD4})-(\ref{eqPD3}) we obtain
\begin{equation}\label{eqPD6}
\liminf\limits_{n\to
\infty}\rho(\pi(t_{n}^{+},x_1^{0}),\pi(t_{n}^{+},x_2^{0}))\ge \rho
(x_1^{0},x_{2}^{0})>0 .
\end{equation}

The relations (\ref{eqPD4}) and (\ref{eqPD6}) are contradictory.
The obtained contradiction proves our statement.
\end{proof}

\begin{definition}\label{defPUS3} A subset $A\subseteq X$ said to be uniformly
attracting \cite[ChIV]{Bro79} if there exists a number
$\delta_{0}>0$ such that for every $\varepsilon
>0$ there exists a number $L(\varepsilon)>0$ with the property
that $a\in A$, $x\in X$, $h(x)=h(a)$ and $\rho(a,x)<\delta_{0}$
imply $\rho(\pi(t,a),\pi(t,x))<\varepsilon$ for all $t\ge
L(\varepsilon)$.
\end{definition}

\begin{lemma}\label{lPUS2} \cite[ChIV]{Bro79} Let $A\subset X$. If
the mapping $h:X\to Y$ is open and the set $A$ is uniformly
attracting, then so is $\overline{A}$.
\end{lemma}

\begin{definition}\label{defPUS4} A subset $A\subset X$ is said to
be uniformly asymptotically stable if it is uniformly positively
stable and uniformly attracting.
\end{definition}

\begin{remark}\label{remPUS1} Let $A\subseteq X$ be a uniformly
attracting subset (respectively, uniformly asymptotically stable)
and $B\subseteq A$, then $B$ is also uniformly attracting
(respectively, uniformly asymptotically stable).
\end{remark}

\begin{definition}\label{defPUS02}
A point $x_0\in X$ is called positively uniformly attracting
(respectively, uniformly asymptotically stable) if the set
$\Sigma_{x_0}^{+}:=\{\pi(t,x_0):\ t\ge 0\}$ is so.
\end{definition}

\begin{coro}\label{lUAS1}\cite[ChIV]{Bro79}
If the point $x_0$ is positively uniformly attracting
(respectively, uniformly asymptotically stable) and $h$ is open,
then:
\begin{enumerate}
\item
 $H^{+}(x_0):=\overline{\Sigma}^{+}_{x_0}$ is positively uniformly attracting (respectively, uniformly asymptotically stable),
 where by bar the closure in $X$ is denoted;
\item
 $\omega_{x_0}$ is positively uniformly attracting (respectively, uniformly asymptotically stable), because $\omega_{x_0}\subseteq H^{+}(x_0)$.
\end{enumerate}
\end{coro}

\begin{lemma}\label{lUAS2} Under the conditions of Theorem
\ref{th02.1} there exists a positive number $r$ such that for
every $q\in \omega_{y_0}$ and $p\in
\omega_{x_0}^{q}=\omega_{x_0}\bigcap X_{q}$ we have $B(p,r)\bigcap
\omega_{x_0}^{q}=\{p\}$, where $B(p,r):=\{x\in X_{q}:\
\rho(x,p)<r\}$.
\end{lemma}
\begin{proof} Since the set $\omega_{x_0}$ is uniformly
attracting then there exists a positive $\delta_{0}$ such that for
all $\varepsilon >0$ there exists a positive number
$L(\varepsilon)$ with the property $\rho(p,x)<\delta_{0}$ ($x\in
X, h(x)=h(p)=q$) imply $\rho(\pi(t,x),\pi(t,p))<\varepsilon$ for
all $t\ge L(\varepsilon)$.

Let now $r\in (0,\delta_{0})$ be an arbitrary number. We will show
that $B(p,r)\bigcap \omega_{x_0}^{q}=\{p\}$ ($h(p)=q$). If we
assume that this statement is false, then there exists a point
$p_0\in \omega_{x_0}^{q}$ such that the set
$\omega_{x_0}^{q}\bigcap B(p_0,r)$ contains a point $p\not= p_0$.
By Theorem \ref{th02.1} there exists a positive number $\alpha
=\alpha(p,p_0)$ such that
\begin{equation}\label{eqPD1}
\rho(\pi(t,p),\pi(t,p_{0}))\ge \alpha
\end{equation}
for all $t\ge 0$. Without loss of generality we can suppose that
$\alpha <r$. On the other hand according to choice of the number
$r$ ($r<\delta_{0}$) for every $\varepsilon \in (0,\alpha)$ there
exists a number $L(\varepsilon)>0$ so that
\begin{equation}\label{eqPD2}
\rho(\pi(t,p),\pi(t,p_0))<\varepsilon <\alpha
\end{equation}
for all $t\ge L(\varepsilon)$. The relations (\ref{eqPD1}) and
(\ref{eqPD2}) are contradictory. The obtained contradiction proves
our statement.
\end{proof}

\begin{coro}\label{corPD1} Under the conditions of Theorem
\ref{th02.1} for every $q\in \omega_{y_0}$ the set
$\omega_{x_0}^{q}$ consists of a finite number $m_{q}$ of points,
i.e., $\omega_{x_0}^{q}=\{p_{1},\ldots,p_{m_{q}}\}$.
\end{coro}
\begin{proof} This statement directly follows from Lemma
\ref{lUAS2} and the fact that the set $\omega_{x_0}^{q}$ is
compact.
\end{proof}

\begin{theorem}\label{thPD2} Let $\langle (X,\mathbb S_{+},\pi ),
(Y,\mathbb S,\sigma ), h\rangle $ be a nonautonomous dynamical
system with the following properties:
\begin{enumerate}
\item there exists a point $x_0\in X$ such that the positive
semi-trajectory $\Sigma_{x_0}^{+}$ is precompact; \item the set
$\omega_{x_0}$ is positively uniformly stable; \item
$\omega_{y_0}$ is a minimal set consisting of almost periodic
points (motions).
\end{enumerate}

Then the following statements hold:
\begin{enumerate}
\item every motions on $\omega_{x_0}$ may be extended uniquely to
the left and on $\omega_{x_0}$ a two-sided dynamical system
$(\omega_{x_0},\mathbb S, \pi)$ is defined; \item the
non-autonomous dynamical system $\langle(\omega_{x_0},$ $\mathbb
S,$ $\pi ),$ $(\omega_{y_0},$ $\mathbb S,$ $\sigma),$ $h\rangle$
is two-sided distal; \item for every $q\in \omega_{y}$ the set
$\omega^{q}_{x}:=\omega_{x}\bigcap X_{q}$ consists of a finite
number $m_{q}$ of points from $\omega_{y}$; \item the number
$m_{q}$ is independent of the point $q\in \omega_{y}$, i.e., there
exist a natural number $m\in \mathbb N$ such that $m_{q}=m$ for
all $q\in \omega_{y}$; \item there exists a number $r>0$ such that
\begin{equation}\label{eqNDS*2}
\rho(\pi(t,p_{i}),\pi(t,p_{j}))\geq r>0\nonumber
\end{equation}
for all $q\in \omega_{y}$, $t\in\mathbb T_{1}$, $p_{i},p_{j}\in
\omega_{x}^{q}$ and $i,j=1,\ldots,m$ with $i\not= j$; \item the
set $\omega_{x}$ consists of the almost periodic points (motions);
\item $\omega_{x}$ is a minimal set of the dynamical system
$(X,\mathbb T_{1},\pi)$; \item the point $x$ (respectively, the
motion $\pi(t,x)$) is remotely almost periodic.
\end{enumerate}
\end{theorem}
\begin{proof}
This statement follows from Theorems \ref{thNDS2}, \ref{th02.1}
and Corollary \ref{corPD1}.
\end{proof}

\begin{coro}\label{corPD2.1} Under the conditions of Theorem
\ref{thPD2} the following statements hold:
\begin{enumerate}
\item if the point $y$ is remotely $\tau$-periodic, then the point
$x$ is remotely $\tau'$-periodic, where $\tau'=m\tau$ and $m$ is
some natural number; \item if the point $y$ is remotely
stationary, then
\begin{enumerate}
\item the point $x$ is asymptotically stationary if the time
$\mathbb T_{1}$ is continuous (i.e., $\mathbb R_{+}\subseteq
\mathbb T_{1}\subseteq \mathbb R$); \item the point $x$ is
asymptotically periodic, if the time $\mathbb T_{1}$ is discrete
(i.e., $\mathbb Z_{+}\subseteq \mathbb T_{1}\subseteq \mathbb Z$).
\end{enumerate}
\end{enumerate}
\end{coro}
\begin{proof}
This statement directly follows from Theorem \ref{thPD2} and
Corollary \ref{corOM1}.
\end{proof}

\begin{remark}\label{remAT1} Note that all results of Sections
\ref{Sec3}--\ref{Sec5} remains true if we replace everywhere the
metric space $Y$ by an uniformly topological (pseudo-metric)
space.
\end{remark}

\section{Applications}\label{Sec6}

Let $\mathfrak B$ and $\mathfrak D$ be two Banach spaces with the
norm $|\cdot|_{\mathfrak B}$ and $|\cdot|_{\mathfrak D}$
respetively\footnote{In the notation $|\cdot|_{\mathfrak B}$
(respectively, $|\cdot|_{\mathfrak D}$), we will omit the index
$\mathfrak B$ (respectively, $\mathfrak D$) if this does not lead
to a misunderstanding.}. Denote by $C(\mathbb T\times \mathfrak
B,\mathfrak D)$ the space of all continuous mappings $f :\mathbb
T\times \mathfrak B\to \mathfrak D$ equipped with the compact-open
topology. For every $h\in \mathbb T$ denote by $f^{h}$ the
$h$-translation of the function $f\in C(\mathbb T\times \mathfrak
B,\mathfrak D)$, i.e., $f^{h}(t,x):=f(t+h,x)$ for all $(t,x)\in
\mathbb T\times \mathfrak B$. The triplet $(C(\mathbb T\times
\mathfrak B,\mathfrak D),\mathbb T,\sigma)$, where $\sigma
:\mathbb T\times C(\mathbb T\times \mathfrak B,\mathfrak D)\to
C(\mathbb T\times \mathfrak B,\mathfrak D)$ is a mapping defined
by the equality $\sigma(h,f):=f^{h}$ ($(h,f)\in \mathbb T\times
C(\mathbb T\times \mathfrak B,\mathfrak D)$), is a shift (or
Bebutov's) dynamical system \cite[Ch.I]{Che_2015}.

\subsection{Ordinary differential equations} \label{Sec6.1}

\begin{example}\label{exODE1}
{\rm Let us consider a differential equations
\begin{equation}
u'=f(t,u),\label{eqAODE1}
\end{equation}
where  $f\in C(\mathbb{R}\times \mathfrak B, \mathfrak B)$.
Everywhere in this subsection we assume that the function $f$ is
regular.

Along with the equation (\ref{eqAODE1}) we consider its $H$-class
(respectively, $H^{+}$-class or $\Omega$-class), i.e., the family
of equations
\begin{equation}
v'=g(t,v),\label{eqAODE2}
\end{equation}
where $g\in H(f):=\overline{\{f^{\tau}:\tau\in \mathbb{R}\}}$
(respectively, $g\in H^{+}(f):=\overline{\{f^{\tau}:\tau\in
\mathbb{R_{+}}\}}$) or
$$
g\in \omega_{f}:=\bigcap\limits_{t\ge
0}\overline{\bigcup\limits_{\tau \ge t}f^{\tau}}.
$$
Denote by $\varphi(\cdot,v,g)$ the solution of equation
(\ref{eqAODE2}), passing through the point $v\in \mathfrak B$ at
the initial moment $t=0$. Then the mapping
$\varphi:\mathbb{R}_{+}\times \mathfrak B\times H(f)\to \mathfrak
B$ (respectively, $\varphi:\mathbb{R}_{+}\times \mathfrak B\times
H^{+}(f)\to \mathfrak B$) is well defined and satisfies the
following conditions (see, e.g. \cite{Bro79,Che_2015,Sel}):
\begin{enumerate}
\item  $\varphi(0,v,g)=v$ for all $v\in \mathfrak B$ and $g\in
H(f)$; \item  $\varphi(t,\varphi(\tau,v,g),
g^{\tau})=\varphi(t+\tau,v,g)$ for all $ v\in \mathfrak B$, $g\in
H(f)$ and $t,\tau \in \mathbb{R}_{+}$;

\item  the mapping $\varphi:\mathbb{R}_{+}\times \mathfrak B\times
H(f)\to \mathfrak B$ is continuous.
\end{enumerate}

Denote by $Y:=H(f)$ (respectively, $Y:=H^{+}(f)$) and
$(Y,\mathbb{R},\sigma)$ (respectively,
$(Y,\mathbb{R}_{+},\sigma)$) the shift dynamical system on $Y$
induced from $(C(\mathbb R\times\mathfrak B,\mathfrak B),\mathbb
R,\sigma)$ (respectively, $(C(\mathbb R_{+}\times\mathfrak
B,\mathfrak B),\mathbb R_{+},\sigma)$), i.e.,
$\sigma(\tau,g)=g^\tau$ for all $\tau\in\mathbb R$ (respectively,
$t\in \mathbb R_{+}$) and $g\in Y$. Then the equation
(\ref{eqAODE1}) generates a cocycle $\langle \mathfrak B,\varphi,
(Y,\mathbb{R},\sigma)\rangle $ and a NDS $\langle
(X,\mathbb{R}_{+},\pi), (Y,\mathbb{R},\sigma), h\rangle$
(respectively, $\langle (X,\mathbb{R}_{+},\pi),
(Y,\mathbb{R}_{+},\sigma), h\rangle$), where $X:= \mathfrak
B\times Y$, $\pi:=(\varphi,\sigma)$ and $h=pr_2:X\to Y$.}
\end{example}

\begin{theorem}\label{thAODE1} Assume that the following
conditions hold:
\begin{enumerate}
\item the equation (\ref{eqAODE1}) admits a precompact on $\mathbb
R_{+}$ solution $\varphi$, i.e., $Q:=\overline{\varphi(\mathbb
R_{+})}$ is a compact subset of $\mathfrak B$; \item the function
$f\in C(\mathbb R\times \mathfrak B,\mathfrak B)$ is positively
Lagrange stable and regular; \item $f$ is remotely almost periodic
(respectively, remotely $\tau$-periodic or remotely stationary) in
$t\in \mathbb T$ uniformly  w.r.t. $x$ on every compact subset of
$\mathfrak B$; \item the $\omega$-limit set $\omega_{f}$ of the
function $f$ is minimal; \item each equation (\ref{eqAODE2})
($g\in \omega_{f}$) has only separated on $\mathbb R_{+}$
solutions in $Q$.
\end{enumerate}

Then the solution $\varphi$ of the equation (\ref{eqAODE1}) is
remotely almost periodic (respectively, remotely $m\tau$-periodic
($m$ is some natural number) or remotely stationary).
\end{theorem}
\begin{proof}
Let $Y:=H^{+}(f)$ and $(Y,\mathbb R_{+},\sigma)$ be the semi-group
dynamical system on $Y$ induced by shift dynamical system
$(C(\mathbb R\times \mathfrak B,\mathfrak B),\mathbb R,\sigma)$.
Since the function $f$ is positively Lagrange stable then the
space $Y$ is compact. Denote by $X:=\mathfrak B\times Y$, let
$(X,\mathbb R_{+},\pi)$ be the skew-product dynamical system ($\pi
:=(\varphi,\sigma)$) and
\begin{equation}\label{eqNDS1}
\langle (X,\mathbb R_{+},\pi),(Y,\mathbb R,\sigma),h\rangle
\end{equation}
the nonautonomous dynamical system associated by the differential
equation (\ref{eqAODE1}) (for more details see Example
\ref{exODE1}). Now to finish the proof of Theorem it suffices to
apply Theorem \ref{thNDS2}, Corollary \ref{corOM1} and Remark
\ref{remAT1} to nonautonomous dynamical system (\ref{eqNDS1}).
\end{proof}

\begin{definition}\label{def_OO1} A solution $\varphi(t,u_0,f)$ of
the equation (\ref{eqAODE1}) is said to be:
\begin{enumerate}
\item[-] positively uniformly Lyapunov stable if for arbitrary
$\varepsilon
>0$ there exists $\delta =\delta(\varepsilon)>0$ such that
$|\varphi(t_0,u,f)-\varphi(t_0,u_0,f)|<\delta$ ($t_0\in \mathbb
R_{+}$, $u\in\mathfrak B$) implies
$|\varphi(t,x,f)-\varphi(t,x_0,f)|<\varepsilon$ for all $t\ge
t_0$; \item[-] uniformly attracting \cite[ChIV]{Bro79} if there
exists a number $\delta_{0}>0$ such that for every $\varepsilon
>0$ there exists a number $L(\varepsilon)>0$ with the property
that $|\varphi(t_0,u,f)-\varphi(t_0,u_0,f)|<\delta_{0}$ ($t_0\ge
0$) imply $|\varphi(t,u,f)-\varphi(t,u_0,f)|<\varepsilon$ for all
$t\ge t_0+L(\varepsilon)$; \item uniformly positively
asymptotically stable if it is positively uniformly stable and
uniformly attracting.
\end{enumerate}
\end{definition}

\begin{theorem}\label{thAODE2} Assume that the following
conditions hold:
\begin{enumerate}
\item the equation (\ref{eqAODE1}) admits a precompact on $\mathbb
R_{+}$ solution $\varphi(t,u_0,f)$; \item the solution
$\varphi(t,u_0,f)$ is uniformly positively asymptotically stable;
\item the function $f\in C(\mathbb R\times \mathfrak B,\mathfrak
B)$ is positively Lagrange stable and regular; \item $f$ is
remotely almost periodic (respectively, remotely $\tau$-periodic
or remotely stationary) in $t\in \mathbb T$ uniformly w.r.t. $x$
on every compact subset of $\mathfrak B$.
\end{enumerate}

Then the solution $\varphi(t,u_0,f)$ of the equation
(\ref{eqAODE1}) is remotely almost periodic (respectively,
remotely $m\tau$-periodic ($m$ is some natural number) or remotely
stationary).
\end{theorem}
\begin{proof}
Let $Y:=H^{+}(f)$ and $(Y,\mathbb R_{+},\sigma)$ be the semi-group
dynamical system on $Y$. Since the function $f$ is positively
Lagrange stable then the space $Y$ is compact. Let $\langle
(X,\mathbb R_{+},\pi),(Y,\mathbb R,\sigma),h\rangle$ be the
nonautonomous dynamical system associated by the differential
equation (\ref{eqAODE1}). Now to finish the proof of Theorem it
suffices to apply Theorem \ref{thPD2}, Corollary \ref{corPD2.1}
and Remark \ref{remAT1} to nonautonomous dynamical system
(\ref{eqNDS1}).
\end{proof}

Let $\mathfrak B =\mathcal H$ be a real Hilbert space with the
scalar product $\langle \cdot,\cdot \rangle =|\cdot |^{2}$, where
$|\cdot|$ is the norm on the Banach space $\mathfrak B$.

Consider a differential equation
\begin{equation}\label{eqH1}
x'=f(t,x),
\end{equation}
where $f\in C(\mathbb T\times \mathcal H,\mathcal H)$.

Assume that the function $f$ satisfies the following condition:

Condition (\textbf{H}): there exist a constant $\alpha >2$ such
that
\begin{equation}\label{eqH2}
\operatorname{Re}\ \langle x_{1}-x_{2}, f(t,x_{1})-f(t,x_{2}) \rangle  \le -
\kappa \vert x_{1}-x_{2} \vert ^{\alpha} \nonumber
\end{equation}
for all $t \in \mathbb T$ and $x_{1},x_{2}\in \mathcal H$ ($\kappa
>0$ and $\alpha >2$).

\begin{remark}\label{remH1} If the function $f\in C(\mathbb T \times
\mathcal H,\mathcal H)$satisfies the Condition (\textbf{H}), then
every function $g\in H(f)$ satisfies the same Condition
(\textbf{H}) with the same constants $\kappa$ and $\alpha$.
\end{remark}

Along with the equation (\ref{eqH1}) we will consider its
$H$-class, i.e., the family of equations
\begin{equation}\label{eqH3}
y'=g(t,y)\ \ (g\in H(f)).
\end{equation}

For every $v\in H$ and $g\in H(f)$ there exists
\cite{TruPer74,TruDas77} (see also \cite[Ch.II]{TruPer86}) a
unique solution $\varphi(t,v,g)$ of the equation (\ref{eqH3})
passing through the point $v\in H$ at the initial moment $t=0$ and
defined on $\mathbb T$.

\begin{lemma}\label{lHD}\cite[Ch.V]{Che_2009}, \cite{CS}
Let $f\in C(\mathbb{R}\times\mathcal{H},\mathcal{H})$ be Lagrange
stable, i.e., the set $\{f^{h}|\ h\in\mathbb{R}\}$ is relatively
compact in $C(\mathbb{R}\times\mathcal{H},\mathcal{H})$ and
\textbf{Condition (H)} is held. Then:
\begin{enumerate}
\item for every $u\in\mathcal{H}$ the solution $\varphi(t,u,f)$ of
the equation $(\ref{eqH1})$ is precompact on $\mathbb{R}_+$ (i.e.,
$\varphi(\mathbb{R}_+,u,f)$ is a precompact subset of
$\mathcal{H}$); \item for all $t\in\mathbb{R}_+$ and
$x_1,x_2\in\mathcal{H}$
$$
|\varphi(t,x_1,f)-\varphi(t,x_2,f)|\leq
(|x_1-x_2|^{2-\alpha}+(\alpha-2) t)^{\frac{1}{2-\alpha}}=
$$
$$
|x_1-x_2|(1+|x_1-x_2|^{\alpha-2}(\alpha-2)
t)^{\frac{1}{2-\alpha}}.
$$
\end{enumerate}
\end{lemma}

\begin{theorem}\label{thH1} Assume that $f\in C(\mathbb R_{+}\times
\mathcal H,\mathcal H)$ and the following conditions are
fulfilled:
\begin{enumerate}
\item the function $f$ satisfies Condition (\textbf{H}); \item the
function $f$ is positively Lagrange stable; \item the function $f$
is remotely stationary (respectively, remotely $\tau$-periodic or
remotely almost periodic) in $t\in \mathbb R_{+}$ uniformly w.r.t.
$x$ on every compact subset from $\mathcal H$; \item the set
$\omega_{f}$ is minimal.
\end{enumerate}

Then every solution $\varphi(t,u,f)$ of the equation (\ref{eqH1})
is remotely stationary (respectively, remotely $\tau$-periodic or
remotely almost periodic).
\end{theorem}
\begin{proof}
Consider the differential equation
\begin{equation}\label{eqfF1}
x'=F(t,x),
\end{equation}
where $F\in C(\mathbb R\times \mathcal H,\mathcal H)$ is the
function defined by the equality
\begin{equation}\label{eqF2}
 F(t,x)=\left\{\begin{array}{ll}
&\!\! f(t,x),\;\mbox{if}\; (t,x)\in \mathbb R_{+}\times \mathcal H \\[2mm]
&\!\! f(0,x),\;\mbox{if}\; (t,x)\in \mathbb R_{-}\times \mathcal
H.\nonumber
\end{array}
\right.\nonumber
\end{equation}

Note that the function $F$ possesses the following properties:
\begin{enumerate}
\item
\begin{equation}\label{eqF3}
\operatorname{Re}\ \langle x_1 -x_2,F(t,x_)-F(t,x_2)\rangle \le -\kappa
|x_1-x_2|^{\alpha}\nonumber
\end{equation}
for all $(t,x_{i}) \in \mathbb R\times \mathcal H$ ($i=1,2$);
\item the function $F$ is Lagrange stable, i.e., the set
$\{F^{h}:\ h\in \mathbb R\}$ is precompact in the space $C(\mathbb
R\times \mathcal H,\mathcal H)$; \item $F$ is remotely stationary
(respectively, remotely $\tau$-periodic or remotely almost
periodic in $t$ uniformly w.r.t. $x$ on every compact subset from
$\mathcal H$.
\end{enumerate}

By Lemma \ref{lHD} we have
\begin{enumerate}
\item[$1)$] for every $u\in\mathcal{H}$ the solution
$\varphi(t,u,F)$ of the equation $(\ref{eqfF1})$ is precompact on
$\mathbb{R}_+$; \item[$2)$] for all $t\in\mathbb{R}_+$ and
$x_1,x_2\in\mathcal{H}$
$$|
|\varphi(t,x_1,F)-\varphi(t,x_2,F)|\leq
(|x_1-x_2|^{2-\alpha}+(\alpha-2) t)^{\frac{1}{2-\alpha}}=
$$
\begin{equation}\label{eqF5.1}
|x_1-x_2|(1+|x_1-x_2|^{\alpha-2}(\alpha-2)
t)^{\frac{1}{2-\alpha}}.
\end{equation}
\end{enumerate}

Note that $H^{+}(F)=H^{+}(f)$ and
\begin{equation}\label{eqF5}
\varphi(h,v,G)=\varphi(h,v,g)
\end{equation}
for all $h\in \mathbb R_{+}$, $v\in \mathcal H$ and $g\in
H^{+}(f)$. From the compactness of $H(F)$ and the relation
(\ref{eqF5}) we obtain that $\varphi(t,u,f)=\varphi(t,u,F)$ ($t\in
\mathbb R_{+}$) is a precompact solution of the equation
(\ref{eqH1}). By (\ref{eqF5.1}) and (\ref{eqF5}) (see also Remark
\ref{remH1}) we have
\begin{equation}\label{eqF4.2}
|\varphi(t,v_1,g)-\varphi(t,v_2,g)|\le \omega(t,|v_1-v_2|)
\end{equation}
for all $t\in \mathbb R_{+}$, $v_1,v_2\in \mathcal H$ and $g\in
H^{+}(f)$, where
\begin{equation}\label{eqF4.3}
\omega(t,r):=r(1+(\alpha -2)tr^{(\alpha
-2)})^{\frac{1}{(2-\alpha)}}\nonumber
\end{equation}
for all $t,r\in \mathbb R_{+}$.

Note that the solution $\varphi(t,u_0,f)$ ($u_{0}\in \mathcal H$)
of the equation (\ref{eqH1}) is precompact on $\mathbb R_{+}$. Now
we will show that it is uniformly asymptotically stable.

Let $\varepsilon$ be an arbitrary positive number and $\delta
(\varepsilon)\in (0,\varepsilon)$. If
$|\varphi(t_0,u,f)-\varphi(t_0,u_0,f)|<\delta(\varepsilon)$, then
by (\ref{eqF4.2}) and taking into account $\omega(t,r)<r$ for all
$t,r>0$ we obtain
$$
|\varphi(t,u,f)-\varphi(t,u_0,f)|=
|\varphi(t-t_0,\varphi(t_0,u,f),f^{t_0})-\varphi(t-t_0,\varphi(t_0,u_0,f),f^{t_0})|\le
$$
$$
\omega(t-t_{0},|\varphi(t_0,u,f)-\varphi(t_0,u_0,f)|)<|\varphi(t_0,u,f)-\varphi(t_0,u_0,f)|<\varepsilon
$$
and, consequently, $|\varphi(t,u,f)-\varphi(t,u_0,f)|<\varepsilon$
for all $t\ge t_0$. Thus the solution $\varphi(t,u_0,f)$ is
positively uniformly stable.

Assume that $\delta_{0}$ is an arbitrary but fixed positive
number, $\varepsilon$ is an arbitrary positive number and $t_0$ is
a nonnegative number such that
$|\varphi(t_0,u,f)-\varphi(t_0,u_0,f)|<\delta_{0}$. By Lemma
\ref{lHD} (item (ii)) we have
\begin{eqnarray}\label{eqF4.4}
& |\varphi(t,u,f)-\varphi(t,u_0,f)|=
|\varphi(t-t_0,\varphi(t_0,u,f),f^{t_0})-\varphi(t-t_0,\varphi(t_0,u_0,f),f^{t_0})|\le
\nonumber \\
& \omega(t-t_0,|\varphi(t_0,u,f)-\varphi(t_0,u_0,f)|)\le
\omega(t-t_0,\delta_0)<\varepsilon\nonumber
\end{eqnarray}
for all $t\ge t_0+L(\varepsilon)$, where
\begin{equation}\label{eqF4.5}
L(\varepsilon):=\frac{\big{[}\big{(}\frac{\delta_{0}}{\varepsilon}\big{)}^{(\alpha
-2)}-1\big{]}}{\delta_{0}(\alpha -2)}.\nonumber
\end{equation}

By Theorem \ref{thAODE2} the solution $\varphi(t,u_0,f)$ is
remotely stationary (respectively, remotely $m\tau$-periodic ($m$
is some natural number) or remotely stationary). Theorem is
proved.
\end{proof}

\begin{example}\label{exH1} {\em Consider a differential equation
\begin{equation}\label{eqHE1}
x'+|x|x=f(t),
\end{equation}
$f\in C(\mathbb R_{+},\mathcal \mathcal H)$ defined by the
equality
\begin{equation}\label{eqHE1.1}
f(t):=\big{(}\sin(t+\ln(1+t))+\sin(\sqrt{2}t+\ln(1+\sqrt{2}t))\big{)}x_0\nonumber
\end{equation}
for all $t\in \mathbb R_{+}$ ($x_0$ is some nonzero element from
$\mathcal H$).}

\begin{lemma}\label{lHE1} The function $f$ possesses the following
properties:
\begin{enumerate}
\item $f$ is positively Lagrange stable; \item $f$ is remotely
almost periodic; \item the $\omega$-limit set $\omega_{f}$ of $f$
is a compact minimal set of the shift dynamical system $(C(\mathbb
R_{+},\mathcal H),\mathbb R_{+},\sigma)$ consisting of almost
periodic functions.
\end{enumerate}
\end{lemma}
\begin{proof} Note that the first two statements are evident.

Since the function $f$ is positively Lagrange stable then the
$\omega$-limit set of $f$ is a nonempty, compact and invariant
subset of $C(\mathbb R_{+},\mathcal H)$. To finish the proof of
the third statement of the theorem we need to show that
$\omega_{f}$ is a minimal set consisting of almost periodic
functions. To this end we will show that $\omega_{f}=\omega_{p}$,
where $p$ is the function $p\in C(\mathbb R,\mathcal H)$ defined
by the equality
$$
p(t):=(\sin t +\sin \sqrt{2}t)x_0
$$
for all $t\in \mathbb
R$. The function $p$ is almost periodic (in fact, quasi-periodic).
Note that the function $\widetilde{p}\in \omega_{p}$ if and only if
there are two numbers $\widetilde{\tau}_{1},\widetilde{\tau}_{2}\in
[0,2\pi)$ such that
$$
\widetilde{p}(t)=(\sin(t+\widetilde{\tau}_{1})+\sin(\sqrt{2}t+\widetilde{\tau}_{2}))x_0
$$
for every $t\in \mathbb R$.

To establish the equality $\omega_{f}=\omega_{p}$ it suffices to
prove that $\omega_{f}\subseteq \omega_{p}$ because $\omega_{f}$
is a compact and invariant set (see Theorem \ref{thLS1}). If
$\widetilde{f}\in \omega_{f}$, then there exists a sequence
$\{h_k\}\subset \mathbb R_{+}$ such that $h_k\to +\infty$ and
$f^{h_k}\to \widetilde{f}$ in the space $C(\mathbb R_{+},\mathcal H)$
as $k\to \infty$.

Note that $h_k=2\pi m_k+\tau^{1}_{k}$ (respectively,
$\sqrt{2}h_k=2\pi l_{k}+\tau_{k}^{2}$) for all $k\in \mathbb N$,
where $m_k\in \mathbb N$ (respectively, $l_k\in \mathbb N)$ and
$\tau^{i}_{k}\in [0,2\pi)$ ($i=1,2$). Without loos of generality
we can suppose that the sequences $\{\tau^{i}_{k}\}$ ($i=1,2$) are
convergent and their limits $\widetilde{\tau}^{i}\in [0,2\pi)$
($i=1,2)$. Since the function $p$ is almost periodic then we may
assume that the sequence $\{p^{h_k}\}$ converges in the space
$C(\mathbb R,\mathcal H)$. Denote its limit by $\widetilde{p}$ then
$\widetilde{p}\in \omega_{p}$ and by reasoning above we have
$\widetilde{p}(t)=\sin(t+\widetilde{\tau}_{1})+\sin (\sqrt{2}t
+\widetilde{\tau}_{2})$ for all $t\in \mathbb R$.

On the other hand we have
$$
\sin(t+h_{k} +\ln (1+t +h_{k}))=\sin(t+h_k +\ln h_{k}
+\ln(1+\frac{1+t}{h_k}))=
$$
\begin{equation}\label{eqP1}
\sin(t+2\pi\bar{m}_k+_{}\bar{\tau}_k^{1}+\ln (1+\frac{1+t}{h_k}))
\end{equation}
and
$$
\sin(\sqrt{2}t+\sqrt{2}h_{k} +\ln (1+\sqrt{2}\sqrt{2}t
+\sqrt{2}h_{k}))=
$$
\begin{equation}\label{eqP2}
\sin(\sqrt{2}t+\sqrt{2}h_k +\ln\sqrt{2} h_{k}
+\ln(1+\frac{1+\sqrt{2}t}{\sqrt{2}h_k}))=
\end{equation}
$$
\sin(\sqrt{2}t+2\pi\bar{l}_k+\bar{\tau}_k^{2}+\ln
(1+\frac{1+\sqrt{2}t}{\sqrt{2}h_k}))
$$
for every $k\in \mathbb N$, where $\{\bar{m}_k\}$ and
$\{\bar{l}_{k}\}$ are two sequences from $\mathbb N$. Without loss
of generality we can assume that the sequences
$\{\bar{\tau}^{i}_{k}\}$ ($i=1,2$) are convergent. Denote by
$\bar{\tau}_{0}^{i}=\lim\limits_{k\to \infty}\bar{\tau}^{i}_{k}$
then we may suppose that $\bar{\tau}_{0}^{i} \in [0,2\pi)$
($i=1,2$). Passing to the limit in the relation (\ref{eqP1})
(respectively, (\ref{eqP2})) and taking into account said above we
obtain
$\widetilde{f}(t)=\sin(t+\bar{\tau}_{0}^{1})+\sin(\sqrt{2}t+\bar{\tau}_{0}^{2})$
for all $t\in \mathbb R$ and, consequently, $\widetilde{f}\in
\omega_{p}$, i.e., $\omega_{f}=\omega_{p}$. Lemma is completely
proved.
\end{proof}

{\em We can rewrite the equation (\ref{eqHE1}) in the form
\begin{equation}\label{eqRHE1}
x'=f(t,x),\nonumber
\end{equation}
where $f\in C(\mathbb R_{+}\times \mathcal H,\mathcal H)$ defined
by the equality $f(t,x)=-|x|x+f(t)$ for all $(t,x)\in \mathbb
R_{+}\times \mathcal H$.

It is easy to see that the function $f$ satisfies Condition
(\textbf{H}) with $ \kappa = \frac{1}{2} $ and $ \alpha = 3. $
Indeed,
$$
\langle x_{1}-x_{2}, f(t,x_{1})-f(t,x_{2}) \rangle = \langle
x_{1}-x_{2}, -x_{1}\vert x_{1}\vert +
 x_{2}\vert x_{2}\vert \rangle
 $$
 $$ = \langle x_{1}-x_{2}, -x_{1}(\vert x_{1}\vert + \vert x_{2}\vert )
 +  x_{2}(\vert x_{1}\vert  +  \vert x_{2}\vert), x_{1}\vert x_{2}\vert -
 x_{2}\vert x_{1}\vert \rangle
 $$
 $$
  = -\vert x_{1}-x_{2} \vert ^{2} (\vert x_{1}\vert +
 \vert x_{2}\vert ) + \frac{1}{2} (\vert x_{1}\vert +  \vert x_{2}\vert )
 (2\vert x_{1}\vert \vert x_{2} \vert - 2 \langle x_{1},x_{2} \rangle).
 $$
Note that
$$
|x_1-x_2|^{2}-(2|x_1||x_2|-2\langle
x_1,x_2\rangle)=(|x_1|-|x_2|)^{2}
$$
and, consequently,
$$
\langle x_{1}-x_{2}, x_{1}\vert
x_{2}\vert -
 x_{2}\vert x_{1}\vert \rangle \le \frac{1}{2} \vert x_{1}-x_{2} \vert ^{2} (\vert x_{1}\vert +
 \vert x_{2}\vert ).
 $$
Thus we have
 $$
 \langle x_{1}-x_{2}, f(t,x_{1})-f(t,x_{2}) \rangle \le -\frac{1}{2} \vert x_{1}-x_{2} \vert ^{2} (\vert x_{1}\vert +
 \vert x_{2}\vert ) \le -\frac{1}{2} \vert x_{1}-x_{2} \vert ^{3}
 $$
for all $x_1,x_2\in \mathcal H$ and $t\in \mathbb R_{+}$. Thus,
Theorem \ref{thH1} is applicable for the equation (\ref{eqHE1}),
i.e., every solution of the equation (\ref{eqHE1}) is remotely
almost periodic. }

\end{example}

\subsection{Difference equations} \label{Sec6.2}

\begin{example}\label{exADE1}
{\rm Let us consider a difference equations
\begin{equation}
u(t+1)=f(t,u(t)),\label{eqADE1}
\end{equation}
where  $f\in C(\mathbb Z_{+}\times \mathfrak B, \mathfrak B)$.
Along with the equation (\ref{eqADE1}) we consider its
$H^{+}$-class (respectively, $\Omega$-class), i.e., the family of
the equations
\begin{equation}
v'=g(t,v),\label{eqADE2}
\end{equation}
where $g\in H^{+}(f):=\overline{\{f^{\tau}:\tau\in \mathbb
Z_{+}\}}$ (respectively, $g\in \omega_{f}:=\bigcap\limits_{t\ge
0}\overline{\bigcup\limits_{\tau \ge t}f^{\tau}})$.

Denote by $\varphi(\cdot,v,g)$ the solution of (\ref{eqADE2}),
passing through the point $v\in \mathfrak B$ at the initial moment
$t=0$. Then the mapping $\varphi:\mathbb{Z}_{+}\times \mathfrak
B\times H^{+}(f)\to \mathfrak B$ is well defined and satisfies the
following conditions (see, e.g. \cite{Bro79,Che_2015,Sel}):
\begin{enumerate}
\item  $\varphi(0,v,g)=v$ for all $v\in \mathfrak B$ and $g\in
H^{+}(f)$; \item  $\varphi(t,\varphi(\tau,v,g),
g^{\tau})=\varphi(t+\tau,v,g)$ for all $ v\in \mathfrak B$, $g\in
H^{+}(f)$ and $t,\tau \in \mathbb{Z}_{+}$;

\item  the mapping $\varphi:\mathbb{Z}_{+}\times \mathfrak B\times
H(f)\to \mathfrak B$ is continuous.
\end{enumerate}

Denote by $Y:=H(f)$ (respectively, $Y:=H^{+}(f)$) and
$(Y,\mathbb{Z_{+}},\sigma)$ (respectively,
$(Y,\mathbb{Z}_{+},\sigma)$) the shift dynamical system on $Y$
induced from $(C(\mathbb Z_{+}\times\mathfrak B,\mathfrak
B),\mathbb Z_{+},\sigma)$ (respectively, $(C(\mathbb
Z_{+}\times\mathfrak B,\mathfrak B),\mathbb Z_{+},\sigma)$), i.e.,
$\sigma(\tau,g)=g^\tau$ for $\tau\in\mathbb Z_{+}$ and $g\in Y$.
Then the equation (\ref{eqADE1}) generates a cocycle $\langle
\mathfrak B,\varphi, (Y,\mathbb Z_{+},\sigma)\rangle $ and a NDS
$\langle (X,\mathbb Z_{+},\pi), (Y,\mathbb Z_{+},\sigma),
h\rangle$, where $X:= \mathfrak B\times Y$,
$\pi:=(\varphi,\sigma)$ and $h=pr_2:X\to Y$.) }
\end{example}

\begin{theorem}\label{thADE1} Assume that the following
conditions hold:
\begin{enumerate}
\item the equation (\ref{eqADE1}) admits a precompact on $\mathbb
Z_{+}$ solution $\varphi$, i.e., $Q:=\overline{\varphi(\mathbb
Z_{+})}$ is a compact subset of $\mathfrak B$; \item the function
$f\in C(\mathbb Z_{+}\times \mathfrak B,\mathfrak B)$ is
positively Lagrange stable; \item $f$ is remotely almost periodic
(respectively, remotely $\tau$-periodic or remotely stationary) in
$t\in \mathbb Z_{+}$ uniformly w.r.t. $x$ on every compact subset
of $\mathfrak B$; \item the $\omega$-limit set $\omega_{f}$ of the
function $f$ is minimal; \item each equation (\ref{eqADE2}) ($g\in
\omega_{f}$) has only separated on $\mathbb Z_{+}$ solutions in
$Q$.
\end{enumerate}

Then the solution $\varphi$ of the equation (\ref{eqADE1}) is
remotely almost periodic (respectively, remotely $m\tau$-periodic
($m$ is some natural number) or remotely stationary).
\end{theorem}
\begin{proof}
Let $Y:=H^{+}(f)$ and $(Y,\mathbb Z_{+},\sigma)$ be the semi-group
dynamical system on $Y$ induced by shift dynamical system
$(C(\mathbb Z_{+}\times \mathfrak B,\mathfrak B),\mathbb
Z_{+},\sigma)$. Since the function $f$ is positively Lagrange
stable then the space $Y$ is compact. Denote by $X:=\mathfrak
B\times Y$, $(X,\mathbb Z_{+},\pi)$ the skew-product dynamical
system ($\pi :=(\varphi,\sigma)$) and
\begin{equation}\label{eqDNDS1}
\langle (X,\mathbb Z_{+},\pi),(Y,\mathbb Z_{+},\sigma),h\rangle
\end{equation}
the nonautonomous dynamical system associated by the difference
equation (\ref{eqADE1}) (for more details see Example
\ref{exADE1}). Now to finish the proof of Theorem it suffices to
apply Theorem \ref{thNDS2}, Corollary \ref{corOM1} and Remark
\ref{remAT1} to nonautonomous dynamical system (\ref{eqDNDS1}).
\end{proof}

\begin{definition}\label{def_DOO1} A solution $\varphi(t,u_0,f)$ of
the equation (\ref{eqADE1}) is said to be:
\begin{enumerate}
\item[-] positively uniformly Lyapunov stable if for arbitrary
$\varepsilon
>0$ there exists $\delta =\delta(\varepsilon)>0$ such that
$|\varphi(t_0,u,f)-\varphi(t_0,u_0,f)|<\delta$ ($t_0\in \mathbb
Z_{+}$, $u\in\mathfrak B$) implies
$|\varphi(t,x,f)-\varphi(t,x_0,f)|<\varepsilon$ for all $t\ge
t_0$; \item[-] uniformly attracting if there exists a number
$\delta_{0}>0$ such that for every $\varepsilon
>0$ there exists a number $L(\varepsilon)>0$ with the property
that $|\varphi(t_0,u,f)-\varphi(t_0,u_0,f)|<\delta_{0}$ ($t_0\in
\mathbb Z_{+}$) imply
$|\varphi(t,u,f)-\varphi(t,u_0,f)|<\varepsilon$ for all $t\ge
t_0+L(\varepsilon)$; \item uniformly positively asymptotically
stable if it is positively uniformly stable and uniformly
attracting.
\end{enumerate}
\end{definition}

\begin{theorem}\label{thADE2} Assume that the following
conditions hold:
\begin{enumerate}
\item the equation (\ref{eqADE1}) admits a precompact on $\mathbb
Z_{+}$ solution $\varphi(t,u_0,f)$; \item the solution
$\varphi(t,u_0,f)$ is uniformly positively asymptotically stable;
\item the function $f\in C(\mathbb Z\times \mathfrak B,\mathfrak
B)$ is positively Lagrange stable and regular; \item $f$ is
remotely almost periodic (respectively, remotely $\tau$-periodic
or remotely stationary) in $t\in \mathbb Z$ uniformly w.r.t. $x$
on every compact subset of $\mathfrak B$.
\end{enumerate}

Then the solution $\varphi(t,u_0,f)$ of the equation
(\ref{eqADE1}) is remotely almost periodic (respectively, remotely
$m\tau$-periodic ($m$ is some natural number) or remotely
stationary).
\end{theorem}
\begin{proof}
Let $Y:=H^{+}(f)$ and $(Y,\mathbb Z_{+},\sigma)$ be the semi-group
dynamical system on $Y$. Since the function $f$ is positively
Lagrange stable then the space $Y$ is compact. Let $\langle
(X,\mathbb Z_{+},\pi),(Y,\mathbb Z_{+},\sigma),h\rangle$ be the
nonautonomous dynamical system associated by the differential
equation (\ref{eqADE1}). Now to finish the proof of Theorem it
suffices to apply Theorem \ref{thPD2}, Corollary \ref{corPD2.1}
and Remark \ref{remAT1} to nonautonomous dynamical system
(\ref{eqDNDS1}).
\end{proof}

\subsection{Functional differential equations with finite delay} \label{Sec6.3}

\begin{example}\label{exFDE1}
{\rm Let us first recall some notions and notations from
\cite{hale}. Let $r>0,\; C([a,b],\mathbb{R}^n)$ be the Banach
space of all continuous functions $\varphi:[a,b]\to \mathbb{R}^n$
equipped with the $\sup$--norm. If $[a,b]=[-r,0]$, then we set $
\mathcal C:=C([-r,0],\mathbb{R}^n)$. Let $\sigma\in\mathbb{R},\,
A\ge 0$ and $u\in C([\sigma-r,\sigma+A],\mathbb{R}^n)$. We will
define $u_t\in \mathcal C$ for every $t\in[\sigma,\sigma+A]$ by
the equality $u_t(\theta):=u(t+\theta),\;-r\le\theta\le0$.
Consider a functional differential equation
\begin{equation}\label{eqAFDE1}
u'(t)=f(t,u_t),
\end{equation}
where $f:\mathbb{R}\times \mathcal C\to\mathbb{R}^n$ is
continuous.

Denote by $C(\mathbb{R}\times \mathcal C,\mathbb R^n)$ the space
of all continuous mappings $f:\mathbb{R}\times \mathcal C\to
\mathbb R^n$ equipped with the compact-open topology (see, e.g.
\cite[ChI]{Che_2015} and \cite[ChI]{Shc72}). On the space
$C(\mathbb{R}\times \mathcal C,\mathbb R^n)$ a shift dynamical
system $(C(\mathbb{R}\times \mathcal C,\mathbb R^n),\mathbb
R,\sigma)$ is defined, where $\sigma(\tau,f):=f^{\tau}$ for all
$(f,\tau)\in C(\mathbb{R}\times \mathcal C,\mathbb R^n)\times \mathbb R$
and $f^{\tau}$ is $\tau$-translation of $f$,
i.e., $f^{\tau}(t,u):=f(t+\tau,u)$ for all $(t,u)\in\mathbb R\times
\mathcal C$. Let us set $H(f):=\overline{\{f^{\tau}:
\tau\in\mathbb{R}\}}$ (respectively,
$H^{+}(f):=\overline{\{f^{\tau}: \tau\in\mathbb{R_{+}}\}}$ and
$\omega_{f}:=\bigcap\limits_{t\ge 0}\overline{\bigcup
\limits_{\tau \ge t}f^{\tau}}$, i.e., $g\in \omega_{f}$ if and
only if there exists a sequence $\{h_k\}\subset \mathbb R_{+}$
such that $h_k\to +\infty$ and $f^{h_k}\to g$ in the space
$C(\mathbb R\times \mathcal C,\mathbb R^{n})$ as $k\to \infty$).

Along with the equation (\ref{eqAFDE1}) let us consider the family
of equations
\begin{equation}\label{eqAFDE2}
v'(t)=g(t,v_t),
\end{equation}
where $g\in H(f)$ (respectively, $g\in H^{+}(f)$ or $g\in
\omega_{f}$).

Condition (\textbf{F}).  The function $f$ is said to be regular
(respectively, positively regular or remotely regular) if for
every equation (\ref{eqAFDE2}) with $g\in H(f)$ (respectively,
$g\in H^{+}(f)$ or $g\in\omega_{f}$) the conditions of existence,
uniqueness and extendability on $\mathbb{R}_{+}$ are fulfilled.

\begin{remark}\label{remF1} \emph{Denote by $\widetilde{\varphi}(t,u,f)$ the solution of
the equation (\ref{eqAFDE1}) defined on $[-r,+\infty)$
(respectively, on $\mathbb{R}$) with the initial condition $u\in
\mathcal C$. By $\varphi(t,u,f)$ we will denote below the
trajectory of equation (\ref{eqAFDE1}), corresponding to the
solution $\widetilde{\varphi}(t,u,f)$, i.e., a mapping from
$\mathbb{R}_{+}$ (respectively, $\mathbb{R}$) into $\mathcal C$,
defined by $\varphi(t,u,f)(s):=\widetilde{\varphi}(t+s,u,f)$ for all
$t\in\mathbb{R}_{+}$ (respectively, $t\in\mathbb{R}$) and $s\in
[-r,0]$. Below we will use the notions of ``solution" and
``trajectory" for equation (\ref{eqAFDE1}) as synonymous concepts,
unless it leads to a misunderstanding.}
\end{remark}

It is well-known \cite{Bro79,Sel} that the mapping $\varphi
:\mathbb R_{+}\times \mathcal C\times H(f) \to \mathcal  C$
possesses the following properties:
\begin{enumerate}
\item $\varphi(0,v,g)=v$ for all $v\in \mathcal C$ and $g\in
H(f)$; \item
$\varphi(t+\tau,v,g)=\varphi(t,\varphi(\tau,v,g),\sigma(\tau,g))$
for every $t,\tau\in\mathbb R_{+}$, $v\in \mathcal C$ and $g\in
H(f)$; \item the mapping $\varphi$ is continuous.
\end{enumerate}
Thus the equation (\ref{eqAFDE1}) generates a cocycle $\langle
\mathcal C,\varphi,(Y,\mathbb R,\sigma)\rangle$ and a
nonautonomous dynamical system $\langle (X,\mathbb R_{+},\pi)$,
$(Y,\mathbb R,\sigma),h\rangle$, where $Y:=H(f)$, $X:=\mathcal
C\times Y$, $\pi := (\varphi,\sigma)$ and $h:=pr_2 :X\to Y$. }
\end{example}

\begin{definition}\label{defCC1}
A function $f\in C(\mathbb T\times \mathcal C,\mathbb R^{n})$ is
called completely continuous if for every bounded subset $A\subset
\mathcal C$ the set $f(\mathbb R\times A)$ is bounded in $\mathbb
R^n$.
\end{definition}

\begin{lemma}\label{lQ01} Let $\varphi(t,u,f)$ be a bounded on $\mathbb
R_{+}$ solution of the equation (\ref{eqAFDE1}). If the function
$f$ is completely continuous, then the set $\varphi(\mathbb
R_{+},u,f)\subset \mathcal C$ is pre-compact in $\mathcal C$.
\end{lemma}
\begin{proof} This statement follows from Lemmas 2.2.3 and 3.6.1
from \cite{hale}.
\end{proof}

\begin{theorem}\label{thAFDE1} Assume that the following
conditions hold:
\begin{enumerate}
\item the function $f\in C(\mathbb T\times \mathcal C,\mathbb
R^{n})$ is completely continuous; \item the equation
(\ref{eqAFDE1}) admits a bounded on $\mathbb R_{+}$ solution
$\varphi$, i.e., $Q:=\overline{\varphi(\mathbb R_{+})}$ is a
compact subset of $\mathbb R^{n}$; \item the function $f\in
C(\mathbb R\times \mathcal C,\mathbb R^{n})$ is positively
Lagrange stable and regular; \item $f$ is remotely almost periodic
(respectively, remotely $\tau$-periodic or remotely stationary) in
$t\in \mathbb T$ uniformly w.r.t. $x$ on every compact subset of
$\mathcal C$; \item the $\omega$-limit set $\omega_{f}$ of the
function $f$ is minimal; \item each equation (\ref{eqAFDE2})
($g\in \omega_{f}$) has only separated on $\mathbb R_{+}$
solutions in $Q$.
\end{enumerate}

Then the solution $\varphi$ of the equation (\ref{eqAFDE1}) is
remotely almost periodic (respectively, remotely $m\tau$-periodic
($m$ is some natural number) or remotely stationary).
\end{theorem}
\begin{proof}
Let $Y:=H^{+}(f)$ and $(Y,\mathbb R_{+},\sigma)$ be the semi-group
dynamical system on $Y$ induced by the shift dynamical system
$(C(\mathbb R\times \mathcal C,\mathbb R^{n}),\mathbb R,\sigma)$.
Since the function $f$ is positively Lagrange stable then the
space $Y$ is compact. Denote by $X:=\mathcal C\times Y$,
$(X,\mathbb R_{+},\pi)$ the skew-product dynamical system ($\pi
:=(\varphi,\sigma)$) and
\begin{equation}\label{eqFNDS1}
\langle (X,\mathbb R_{+},\pi),(Y,\mathbb R,\sigma),h\rangle
\end{equation}
the nonautonomous dynamical system associated by the differential
equation (\ref{eqAFDE1}) (for more details see Example
\ref{exFDE1}). Now to finish the proof of Theorem it suffices to
apply Theorem \ref{thNDS2}, Corollary \ref{corOM1} and Remark
\ref{remAT1} to nonautonomous dynamical system (\ref{eqFNDS1}).
\end{proof}

\begin{definition}\label{def^OO1} A solution $\varphi(t,u_0,f)$ of
the equation (\ref{eqAFDE1}) is said to be:
\begin{enumerate}
\item[-] positively uniformly stable if for arbitrary $\varepsilon
>0$ there exists $\delta =\delta(\varepsilon)>0$ such that
$|\varphi(t_0,u,f)-\varphi(t_0,u_0,f)|<\delta$ ($t_0\in \mathbb
R_{+}$, $u\in\mathfrak B$) implies
$|\varphi(t,x,f)-\varphi(t,x_0,f)|<\varepsilon$ for all $t\ge
t_0$; \item[-] uniformly attracting if there exists a number
$\delta_{0}>0$ such that for every $\varepsilon
>0$ there exists a number $L(\varepsilon)>0$ with the property
that $|\varphi(t_0,u,f)-\varphi(t_0,u_0,f)|<\delta_{0}$ ($t_0\ge
0$) imply $|\varphi(t,u,f)-\varphi(t,u_0,f)|<\varepsilon$ for all
$t\ge t_0+L(\varepsilon)$; \item uniformly positively
asymptotically stable if it is positively uniformly stable and
uniformly attracting.
\end{enumerate}
\end{definition}

\begin{theorem}\label{thAFDE2} Assume that the following
conditions hold:
\begin{enumerate}
\item the function $f\in C(\mathbb T\times \mathcal C,\mathbb
R^{n})$ is completely continuous; \item the equation
(\ref{eqAFDE1}) admits a bounded on $\mathbb R_{+}$ solution
$\varphi(t,u_0,f)$; \item the solution $\varphi(t,u_0,f)$ is
uniformly positively asymptotically stable; \item the function
$f\in C(\mathbb R\times \mathcal C,\mathbb R^{n})$ is positively
Lagrange stable and regular; \item $f$ is remotely almost periodic
(respectively, remotely $\tau$-periodic or remotely stationary) in
$t\in \mathbb T$ uniformly w.r.t. $x$ on every compact subset of
$\mathcal C$.
\end{enumerate}

Then the solution $\varphi(t,u_0,f)$ of the equation
(\ref{eqAFDE1}) is remotely almost periodic (respectively,
remotely $m\tau$-periodic ($m$ is some natural number) or remotely
stationary).
\end{theorem}
\begin{proof}
Let $Y:=H^{+}(f)$ and $(Y,\mathbb R_{+},\sigma)$ be the semi-group
dynamical system on $Y$. Since the function $f$ is positively
Lagrange stable then the space $Y$ is compact. Let $\langle
(X,\mathbb R_{+},\pi),(Y,\mathbb R,\sigma),h\rangle$ be the
nonautonomous dynamical system associated by the differential
equation (\ref{eqAFDE1}) (see Example \ref{exFDE1}). Now to finish
the proof of Theorem it suffices to apply Theorem \ref{thPD2},
Corollary \ref{corPD2.1} and Remark \ref{remAT1} to nonautonomous
dynamical system (\ref{eqFNDS1}).
\end{proof}

\subsection{Ordinary Stochastic Differential Equations}

\begin{example}\label{exSDE1}
{\rm Let $(\Omega, \mathcal{F},  {P} )$ be a probability space,
and ${L}^{2}( {P}, \mathbb R^n)$ stand for the space of all
$\mathbb R^n$-valued random variables $X$ such that
\begin{equation}
{E} |X|^{2} = \int_{\Omega}|X|^{2}d  { P}<\infty,\nonumber
\end{equation}
where $|\cdot|$ denotes the Euclidian norm of a vector in $\mathbb
R^n$. For $X\in {L}^{2}( {P},\mathbb R^n)$, let
\[
\|X\| {:=}\left( \int_{\Omega}|X|^{2}d  {P} \right)^{ 1/2 }.
\]
Then ${L}^{2} ( {P},\mathbb R^n)$ is a Hilbert space equipped with
the norm $\|\cdot\|$.

Let $\mathcal P(\mathbb R^n)$ be the space of all Borel
probability measures on $\mathbb R^n$ endowed with the $\beta$
metric:
\begin{equation}\label{eqBT1}
\beta (\mu,\nu) :=\sup\left\{ \left| \int f d \mu - \int f d
\nu\right|: \|f\|_{BL} \le 1 \right\},\ \mu,\nu\in \mathcal
P(\mathbb R^n),\nonumber
\end{equation}}
where $f$ are Lipschitz continuous real-valued functions on
$\mathbb R^n$ with the norms
$$
\|f\|_{BL}= \|f\|_L + \|f\|_\infty, \|f\|_L=\sup_{x\neq y}
\frac{|f(x)-f(y)|}{|x-y|},
$$
$$
 \|f\|_{\infty}=\sup_{x\in \mathbb R^n}|f(x)|.
$$
Recall \cite[Ch.IX]{Dud_2004} that a sequence $\{\mu_k\}\subset
\mathcal P(\mathbb R^n)$ is said to weakly converge to $\mu$ if
$\int f d\mu_k\to \int f d \mu$ for all $f\in C_b(\mathbb R^n)$,
where $C_b(\mathbb R^n)$ is the space of all bounded continuous
real-valued functions on $\mathbb R^n$.

\begin{definition}\label{aad}
A sequence of random variables $\{X_k\}$ is said to \emph{converge
in distribution} to the random variable $X$ if the corresponding
laws $\{\mu_k\}$ weakly converge to the law $\mu$ of $X$, i.e.
$\beta(\mu_k,\mu)\to 0$.
\end{definition}

\begin{definition}
Consider the stochastic differential equation on $\mathbb R^n$:
\begin{equation}\label{eqSDE1}
d X(t) = f(t,X(t))dt + g(t,X(t))dW(t),\nonumber
\end{equation}
where $W(t)$ is a two-sided standard one-dimensional Brown motion
defined on the filtered probability space $(\Omega,\mathcal
F,\mathbb P,\mathcal F_{t})$, where $\mathcal
F_{t}:=\sigma\{W(u)-W(v):\ u,v\le t\}$.

An $\mathcal F_t$-progressively measurable stochastic process
$\{X(t)\}_{t\in \mathbb R_{+}}$ is called a solution of the above
equation if it satisfies the corresponding stochastic integral
equation
\begin{equation*}
X(t)=  X(r)+\int_r^t f(s,X(s))d s +\int_r^t g(s,X(s))dW(s)
\end{equation*}
for all $t\ge r$ and each $r\in \mathbb R$, where $\mathcal
F_t:=\sigma\{W(s): s\le t\}$.
\end{definition}

\begin{definition}\label{defLK1} We say that the
functions $f,g\in C(\mathbb T\times \mathbb R^{n},\mathbb R^{n})$
satisfy the Condition:
\begin{enumerate}
\item[(\textbf{C1}):] if there exists a positive number $L$ such
that
\begin{equation}\label{eqLK1}
|f(t,x)-f(t,y)|\vee |g(t,x)-g(t,y)|\le L|x-y|\nonumber
\end{equation}
for all $t\in \mathbb T$ and $x,y\in \mathbb R^{n}$, where $a\vee
b :=\max\{a,b\}$; \item[(\textbf{C2}):] if there exists a positive
constant $K$ such that
\begin{equation}\label{eqLK2}
|f(t,x)|\vee |g(t,x)|\le K(1+|x|)\nonumber
\end{equation}
for all $(t,x)\in \mathbb T\times \mathbb R^{n}$.
\end{enumerate}
\end{definition}

\begin{remark}\label{remC1}
Note that if the function $(f,g)\in C(\mathbb T\times \mathbb
R^{n},\mathbb R^{n})\times C(\mathbb T\times \mathbb R^{n},\mathbb
R^{n})$ satisfies Condition (\textbf{C1}) (respectively, Condition
(\textbf{C2})), then every function $(\widetilde{f}.\widetilde{g})\in
H(f,g)$ satisfies the same condition with the same constant $L$
(respectively, $K$).
\end{remark}

Consider the stochastic differential equation on $\mathbb R^n$
\begin{equation}\label{eqFG1}
dX = f(t,X)dt + g(t,X)d W.
\end{equation}
where $f,g:\mathbb R\times \mathbb R^n\to \mathbb R^n$ are
continuous and satisfy global Lipschitz (Con\-di\-tion
(\textbf{C1})) and linear growth conditions in $x$ (Condition
(\textbf{C2})) uniformly in $t\in \mathbb R$.

Along with the equation (\ref{eqFG1}) we will consider its
$H$-class, i.e., the family of the equations
\begin{equation}\label{eqFG2}
dX = \widetilde{f}(t,X)dt + \widetilde{g}(t,X)d W,
\end{equation}
where $(\widetilde{f},\widetilde{g})\in H(f,g)$.

Define the mapping
\[
\Phi: \mathbb R_{+} \times \mathcal P(\mathbb R^n) \times H(f,g)
\to \mathcal P(\mathbb R^n),
\]
with $\Phi(t,\mu_0,(\widetilde f,\widetilde g))$ being the law (or
distribution) $\mathcal L(X(t))$ of the solution $X(\cdot)$ at
time $t$ of the Cauchy problem
\begin{equation}\label{tilfg}
d X = \widetilde f(t,X) d t + \widetilde g(t,X) d W, \quad
X(0)=X_0,\nonumber
\end{equation}
where $\mathcal L(X_0)=\mu_0$ and $\Phi(0,\mu_0,(\widetilde f,\widetilde
g))=\mu_0$.

We have the following basic result on $\Phi$:

\begin{theorem}\label{cocycle}\cite{CL_2024}
The mapping $\Phi$ is a continuous cocycle with base space
$H(f,g)$ and fiber space $\mathcal P(\mathbb R^n)$, i.e., the
mapping $\Phi: \mathbb R_{+} \times \mathcal P(\mathbb R^n) \times
H(f,g) \to \mathcal P(\mathbb R^n)$ is continuous and satisfies
$$
\Phi(0,\mu, (\widetilde f,\widetilde g))=\mu,
$$
$$
\Phi(t+\tau, \mu, (\widetilde f,\widetilde g))= \Phi(t,\Phi(\tau,\mu,
(\widetilde f,\widetilde g)), (\widetilde f_\tau,\widetilde g_\tau))
$$
for all $t,\tau\ge 0$, $(\widetilde f, \widetilde g)\in H(f,g)$ and
$\mu\in\mathcal P(\mathbb R^n)$.
\end{theorem}

\begin{coro}
The mapping given by
$$
\Pi:\mathbb R_{+}  \times \mathcal P(\mathbb R^n) \times H(f,g)
\to \mathcal P(\mathbb R^n) \times H(f,g),
$$
$$ \Pi (t, (\widetilde f,\widetilde g), \mu) = (\Phi(t,\mu,(\widetilde f,\widetilde
g)), (\widetilde f_t,\widetilde g_t))
$$
is a continuous skew-product semiflow.
\end{coro}

Thus under the Conditions (\textbf{C1}) and (\textbf{C2}) the
equation (\ref{eqFG1}) generates a cocycle $\langle \mathcal
P(\mathbb R^{n}),\Phi, (H(f,g),\mathbb T,\sigma)\rangle$,
skew-product dynamical system $(X,\mathbb R_{+},\pi)$
($X:=\mathcal P(\mathbb R^{n})\times H(f,g)$, $\pi
:=(\Phi,\sigma)$) and non-autonomous dynamical system $\langle
(X,\mathbb R_{+},\pi),$ $(Y,\mathbb T,\sigma),h\rangle$
($Y:=H(f,g)$ and $h:=pr_{2}:X\to Y$).
\end{example}

\begin{definition}\label{defSDE1}  \textit{Let $\varphi :\mathbb
T\to \mathbb R^{n}$ be a mild solution of the equation
(\ref{eqFG1}). Then $\varphi$ is called almost periodic
(respectively, remotely almost periodic) {\em in distribution} if
the function $\phi \in C(\mathbb T,\mathcal P(\mathbb R^{n}))$ is
almost periodic (respectively, remotely almost periodic), where
$\phi(t):=\mathcal L(\varphi(t))$ for every $t\in\mathbb T$ and
$\mathcal L(\varphi(t))\in \mathcal P(\mathbb R^{n})$ is the law
of random variable $\varphi(t)$.}
\end{definition}

\begin{lemma}\label{lLK1} Assume that the functions $f,g\in C(\mathbb T\times \mathbb R^{n},\mathbb
R^{n})$ are Lagrange stable and $\varphi :\mathbb T\to \mathbb
R^{n}$ is a bounded solution of the equation (\ref{eqSDE1}). If
the functions $f$ and $g$ satisfy Conditions $(\textbf{C1})$ and
$(\textbf{C2})$, then the solution $\varphi$ is precompact in
distribution, i.e., $Q:=\overline{\phi(\mathbb T)}$ is a compact
subset of $\mathcal P(\mathbb R^{n})$.
\end{lemma}
\begin{proof} This statement directly follows from Theorem 3.1
\cite{LW_2016}.
\end{proof}

Let $Q$ be a nonempty compact subset of $\mathbb R^{n}$.

\begin{definition}\label{defSD1} \cite{LW_2016} A solution $\varphi :\mathbb R\to Q$ of
the equation (\ref{eqFG1}) is said to be semi-separated in
distribution in $Q:=\overline{\phi(\mathbb R)}$
($\phi(t):=\mathcal L(\varphi(t))$ for every $t\in \mathbb R$), if
there exists a number $d(\varphi)>0$ so that if $\psi :\mathbb
R\to Q$ is an other solution of (\ref{eqFG1}), then
\begin{equation}\label{eqSD2}
\beta(\mathcal L(\varphi(t)),\mathcal L(\psi(t)))\ge
d(\varphi)\nonumber
\end{equation}
for all $t\in \mathbb R_{+}$.
\end{definition}

\begin{theorem}\label{thSODE1} Assume that the following
conditions hold:
\begin{enumerate}
\item the function $(f,g)\in C(\mathbb R\times \mathbb
R^{n},\mathbb R^{n})\times C(\mathbb R\times \mathbb R^{n},\mathbb
R^{n})$ satisfies Conditions (\textbf{C1}) and (\textbf{C2});
\item the equation (\ref{eqFG1}) admits a bounded on $\mathbb
R_{+}$ solution $\varphi$, i.e., $\overline{\varphi(\mathbb
R_{+})}$ is a compact subset of $\mathbb R^{n}$; \item the
function $(f,g)\in C(\mathbb R\times \mathbb R^{n},\mathbb
R^{n})\times C(\mathbb R\times \mathbb R^{n},\mathbb R^{n})$ is
positively Lagrange stable; \item the functions $f$ and $g$ are
remotely almost periodic (respectively, remotely $\tau$-periodic
or remotely stationary) in $t\in \mathbb T$ uniformly w.r.t. $x$
on every compact subset of $\mathbb R^{n}$; \item the
$\omega$-limit set $\omega_{(f,g)}$ of the function $(f,g)\in
C(\mathbb R\times \mathbb R^{n},\mathbb R^{n})\times C(\mathbb
R\times \mathbb R^{n},\mathbb R^{n})$ is minimal; \item each
equation (\ref{eqFG2}) ($(\widetilde{f},\widetilde{g})\in \Omega_{(f,g)}$)
has only semi-separated in distribution solutions in $Q$ defined
on $\mathbb R$.
\end{enumerate}

Then the solution $\varphi$ of the equation (\ref{eqFG1}) is
remotely almost periodic (respectively, remotely $m\tau$-periodic
($m$ is some natural number) or remotely stationary).
\end{theorem}
\begin{proof}
Let $Y:=H^{+}(f,g)$ and $(Y,\mathbb R_{+},\sigma)$ be the
semi-group dynamical system on $Y$ induced by shift dynamical
system $(C(\mathbb R\times \mathbb R^{n},\mathbb R^{n})\times
C(\mathbb R\times \mathbb R^{n},\mathbb R^{n}),\mathbb R,\sigma)$.
Since the function $(f,g)$ is positively Lagrange stable then the
space $Y$ is compact. Denote by $X:=\mathcal P(\mathbb
R^{n})\times Y$, $(X,\mathbb R_{+},\pi)$ the skew-product
dynamical system ($\pi :=(\Phi,\sigma)$) and
\begin{equation}\label{eqSNDS1}
\langle (X,\mathbb R_{+},\pi),(Y,\mathbb R,\sigma),h\rangle
\end{equation}
the nonautonomous dynamical system associated by the differential
equation (\ref{eqFG1}) (for more details see Example
\ref{exSDE1}). Now to finish the proof of Theorem it suffices to
apply Theorem \ref{thNDS2}, Corollary \ref{corOM1} and Remark
\ref{remAT1} to nonautonomous dynamical system (\ref{eqSNDS1}).
\end{proof}

\begin{definition}\label{def_FG1} A solution $\varphi$ of
the equation (\ref{eqFG1}) is said to be:
\begin{enumerate}
\item[-] positively uniformly stable in distribution if for
arbitrary $\varepsilon
>0$ there exists $\delta =\delta(\varepsilon)>0$ such that
$\beta(\mathcal L(\varphi(t_0)),\mathcal L(\psi(t_0))|<\delta$
($t_0\in \mathbb R_{+}$) implies $\beta(\mathcal
L(\varphi(t)),\mathcal L(\psi(t))|<\varepsilon$ for every $t\ge
t_0$; \item[-] uniformly attracting if there exists a number
$\delta_{0}>0$ such that for every $\varepsilon
>0$ there exists a number $L(\varepsilon)>0$ with the property
that $\beta(\mathcal L(\phi(t_0)),\mathcal
L(\psi(t_0))|<\delta_{0}$ ($t_0\ge 0$) implies $\beta(\mathcal
L(\varphi(t)),\mathcal L(\psi(t))|<\varepsilon$ for all $t\ge
t_0+L(\varepsilon)$; \item[-] uniformly positively asymptotically
stable if it is positively uniformly stable and uniformly
attracting.
\end{enumerate}
\end{definition}

\begin{theorem}\label{thSODE2} Assume that the following
conditions hold:
\begin{enumerate}
\item the function $(f,g)\in C(\mathbb R\times \mathbb
R^{n},\mathbb R^{n})\times C(\mathbb R\times \mathbb R^{n},\mathbb
R^{n})$ satisfy the Conditions $(\textbf{C1})$ and
$(\textbf{C2})$; \item the equation (\ref{eqFG1}) admits a bounded
on $\mathbb R_{+}$ solution $\varphi$; \item the solution
$\varphi$ is uniformly positively asymptotically stable; \item the
function $(f,g)\in C(\mathbb R\times \mathbb R^{n},\mathbb
R^{n})\times C(\mathbb R\times \mathbb R^{n},\mathbb R^{n})$ is
positively Lagrange stable; \item $(f,g)$ is remotely almost
periodic (respectively, remotely $\tau$-periodic or remotely
stationary) in $t\in \mathbb T$ uniformly w.r.t. $x$ on every
compact subset of $\mathbb R^{n}$.
\end{enumerate}

Then the solution $\varphi$ of the equation (\ref{eqFG1}) is
remotely almost periodic (respectively, remotely $m\tau$-periodic
($m$ is some natural number) or remotely stationary).
\end{theorem}
\begin{proof}
Let $Y:=H^{+}(f,g)$ and $(Y,\mathbb R_{+},\sigma)$ be the
semi-group dynamical system on $Y$. Since the function $(f,g)$ is
positively Lagrange stable then the space $Y$ is compact. Let
$\langle (X,\mathbb R_{+},\pi),(Y,\mathbb R,\sigma),h\rangle$ be
the nonautonomous dynamical system associated by the differential
equation (\ref{eqFG1}). Now to finish the proof of Theorem it
suffices to apply Theorem \ref{thPD2}, Corollary \ref{corPD2.1}
and Remark \ref{remAT1} to nonautonomous dynamical system
(\ref{eqSNDS1}).
\end{proof}

\subsection{Algebraic equations} \label{Sec6.4}

The aim of this subsection is to study the problem of existence of
remotely almost periodic solutions of algebraic equations
\begin{equation}\label{eqAE0}
 x^{n}+a_1(t) x^{n-1}+\ldots +a_{n-1}(t) x+a_{n}(t)=0
 \end{equation}
with the remotely almost periodic coefficients $a_{i}\in C(\mathbb
R,\mathbb C)\ (i=1,\ldots,n)$, i.e., we study the problem
algebraic closeness of the space of remotely almost periodic
functions, where $\mathbb C$ is the set of all complex numbers.

As was shown in the papers of Walter \cite{Wal_1933} and Bohr and
Flanders \cite{BF_1937}, if the discriminant
$$
D(t)=D[a_{1}(t),\ldots,a_{n}(t)]
$$
of the equation (\ref{eqAE0}) is separated from zero, i.e.,
$$
\inf_{t\in \mathbb R}|D(t)|
>0,
$$
then the continuous solutions $X=X(t)$ of the equation
(\ref{eqAE0}) with the almost periodic coefficients $a_{i}\in
C(\mathbb R,\mathbb C)\ (i=1,\ldots,n)$ are almost-periodic
functions.

Zhikov (see, for example, \cite[Example 2, pp.111-112]{Lev-Zhi})
constructed an example of almost-periodic function $f$, given on
the real line, satisfying the conditions
$$
|f(t)|>0, \ \inf\limits_{t\in \mathbb R} |f(t)|=0,
$$
and such that the solutions $X=X(t)$ to the equation
$$
X^{2}-f(t) = 0
$$
are not almost-periodic. This means that in the
Walter--Bohr--Flanders Theorem, the condition of separation from
zero cannot be replaced by the condition
$$
|D(t)|>0\ (\forall \ t\in \mathbb R).
$$
Gorin and Lin, in their joint paper \cite{GL_1969}, considered the
equation (\ref{eqAE0}) with almost-periodic coefficients, given on
an arbitrary group, and obtained a deep generalization of the
Walter--Bohr-- Flanders Theorem. Bronshtein \cite{Bro_73}
generalized Gorin-Lin results for algebraic equations with
recurrent coefficients.

Let
\begin{equation}\label{eqAE1}
 x^{n}+a_1 x^{n-1}+\ldots +a_{n-1} x+a_{n}=0
 \end{equation}
be an algebraic equation of degree $n$ with the coefficients
$a_1,\ldots,a_n\in \mathbb C$. By fundamental theorem of algebra
\cite[Ch.V]{Kur_1968} the equation (\ref{eqAE1}) has exactly $n$
solutions  $\lambda_{1},\ldots,\lambda_{n}\in \mathbb C$ (counting
multiplicities). The discriminant $D=D(a_1,\ldots,a_n)$ of the
equation (\ref{eqAE1}) is defined by
\begin{equation}\label{eqDE1}
D(a_1,\ldots,a_n):=\prod_{i\not=
j}(\lambda_{i}-\lambda_{j}).\nonumber
\end{equation}
It is well known \cite[Ch.VI]{Kos_1982} that the discriminant
$D(a_1,\ldots,a_n)$ of the equation (\ref{eqAE1}) is a polynomial
of the coefficients $a_1,\ldots,a_n$.

\begin{lemma}\label{lAE1}\cite[Ch. V]{Kur_1968} We have the following inequality
\begin{equation}\label{eqAE2}
|\lambda_{i}|\le r:=1+A \nonumber
\end{equation}
for every $i=1,\ldots,n$, where $A:=\max\{|a_1|,\ldots,|a_n|\}$.
\end{lemma}

Let $\mathbb T\in \{\mathbb R_{+}, \mathbb R\}$ and $a_{k}\in
C(\mathbb T,\mathbb C)$ ($k=1,\ldots,n$). Consider algebraic
equations of the form (\ref{eqAE0}).

\begin{definition}\label{defSAE1} A continuous function $\lambda :\mathbb T\to \mathbb
C$ which satisfies equation (\ref{eqAE0}) identically is called a
solution of the algebraic equation (\ref{eqAE0}).
\end{definition}

\begin{remark}\label{remSAE1} If the discriminant
\begin{equation}\label{eqDT1}
D(t):=D(a_1(t),\ldots,a_{n}(t))\nonumber
\end{equation}
of the equation (\ref{eqAE0}) is not equal to zero for all $t\in
\mathbb T$, then there exist exactly $n$ solution of (\ref{eqAE0})
\cite{GL_1969}.
\end{remark}

Let $a\in C(\mathbb T,\mathbb C)$ be a Lagrange stable function,
i.e., the set $\Sigma_{a}:=\{a^{h}:\ h\in \mathbb T\}$, where
$a^{h}(t):=a(t+h)$ ($t\in \mathbb T$), is precompact in the space
$C(\mathbb T,\mathbb C)$.

\begin{lemma}\label{lA1.0} \cite{Sel,sib} The following statements
are equivalent:
\begin{enumerate}
\item the function $a\in C(\mathbb T,\mathbb C)$ is Lagrange
stable; \item the function $a$ is bounded and uniformly continuous
on $\mathbb T$.
\end{enumerate}
\end{lemma}

\begin{theorem}\label{thDC1} \cite[Theorem 3]{DP_1964} Let $a_1,\ldots,a_n\in C(\mathbb T,\mathbb C)$. Then there exists
at least one solution $\lambda \in C(\mathbb T,\mathbb C)$ of the
equation (\ref{eqAE0}).
\end{theorem}

\begin{lemma}\label{lA2}
Let $a_1,\ldots,a_n\in C(\mathbb T,\mathbb C)$ be Lagrange stable
functions and $\lambda :\mathbb T\to \mathbb C$ be a continuous
solution of the equation (\ref{eqAE0}).

Then the solution $\lambda$ is bounded.
\end{lemma}
\begin{proof} Since the functions $a_{i}\in C(\mathbb T,\mathbb
C)$ ($i=1,\ldots,n$) are Lagrange stable then by Lemma \ref{lAE1}
there exists a positive number $C$ such that
\begin{equation}\label{eqAE_5}
|a_{i}(t)|\le C
\end{equation}
for all $i=1,\ldots,n$ and $t\in \mathbb T$. By Lemma \ref{lAE1}
we have
\begin{equation}\label{eqAE6}
|\lambda(t)|\le 1+A(t)
\end{equation}
for every $t\in \mathbb T$, where
\begin{equation}\label{eqAE7}
A(t):=\max  x\{|a_{1}(t)|,\ldots,|a_{n}(t)|\}.
\end{equation}
According to (\ref{eqAE_5}) and (\ref{eqAE7}) we obtain
\begin{equation}\label{eqAE8}
A(t)\le C
\end{equation}
for every $t\in \mathbb T$. From (\ref{eqAE6}) and (\ref{eqAE8})
we receive
\begin{equation}\label{eqAE9}
|\lambda(t)|\le 1+C \nonumber
\end{equation}
for all $t\in \mathbb T$.
\end{proof}

\begin{example}\label{exA1}
{\rm Let $a_1,\ldots,a_n\in C(\mathbb T,\mathbb C)$ be Lagrange
stable functions and $\lambda\in C(\mathbb T,\mathbb C)$ be a
continuous solution of the equation (\ref{eqAE0}).

Along with the equation (\ref{eqAE0}) we consider its $H$-class
(respectively, $\Omega$-class), i.e., the family of equations
\begin{equation}\label{eqAE3.1}
x ^{n}+b_{1}(t)x ^{n-1}+\ldots +b_{n-1}(t)x+b_{n}(t)=0, \nonumber
\end{equation}
where $b=(b_1,\ldots,b_{n})\in H(a)$ (respectively, $b\in
\Omega_{a}$).

Denote by $(H(a),\mathbb T,\sigma)$ (respectively, by
$(\Omega_{a},\mathbb T,\sigma)$) the shift dynamical system on
$H(a)$ (respectively, on $\Omega_{a}$).

Let $\lambda :\mathbb T\to \mathbb C$ be a continuous solution of
the equation (\ref{eqAE0}). Denote by
$X:=H(\lambda,a):=\overline{\{(\lambda^{h},a^{h}):\ h\in \mathbb
T\}}$ (respectively, $X:=\Omega_{(\lambda,a)}$) and $(X,\mathbb
T,\pi)$ be the shift dynamical system on the set $X$ induced by
the product dynamical system $(C(\mathbb T,\mathbb C)\times
C(\mathbb T,\mathbb C^{n}))$. Finally we put $h:=pr_{2}:X \to Y$
then the triplet $\langle (X,\mathbb T,\pi),(Y,\mathbb
T,\sigma),h\rangle$ is a nonautonomous dynamical system associated
by algebraic equation (\ref{eqAE0}) \cite{Bro_73}, \cite{Zhi69}.}
\end{example}

\begin{lemma}\label{lD1} Let $\mathbb T\in \{\mathbb R_{+},\mathbb R\}$, $a_1,\ldots,a_n\in C(\mathbb T,\mathbb
C)$ and
$$
D(t)=D(a_1(t),\ldots,a_n(t))
$$
be the discriminant of the
equation (\ref{eqAE0}). Assume that $D(t)$ is separated from zero
on $\mathbb T$, i.e.,
\begin{equation}\label{eqD1}
\gamma:=\inf\limits_{t\in \mathbb T} |D(t)| > 0.
\end{equation}

Then for every $t\in \mathbb T$ and $(b_1,\ldots,b_n)\in
H(a_1,\ldots,a_n)$ we have
\begin{equation}\label{eqD2}
 |\widetilde{D}(t)|\ge \gamma ,
\end{equation}
where $\widetilde{D}(t):=D(b_1(t),\ldots,b_{n}(t))$.
\end{lemma}
\begin{proof} Let $(b_1,\ldots,b_n)\in H(a_1,\ldots,a_n)$ then
there exists a sequence $\{h_k\}\subset \mathbb T$ such that
\begin{equation}\label{eqD3}
b_{i}=\lim\limits_{k\to \infty}a_{i}^{h_k}\ \ (i=1,\ldots,n)
\end{equation}
in the space $C(\mathbb T,\mathbb C)$, where
$a_{i}^{h}(t):=a_{i}(t+h)$ for all $t\in \mathbb T$ and
$i=1,\ldots,n$. Taking into account that $D(a_1,\ldots,a_n)$ is a
polynomial of the coefficients $a_1,\ldots,a_n$ and the relations
(\ref{eqD2}) and (\ref{eqD3}) we obtain
\begin{equation}\label{eqD4}
|\widetilde{D}(t)|=\lim\limits_{k\to
\infty}|D(a_1(t+h_k),\dots,a_n(t+h_k))|\ge \gamma \nonumber
\end{equation}
for all $t\in \mathbb T$. Lemma is proved.
\end{proof}

\begin{lemma}\label{lD1.1} Let the functions $a_1,\ldots,a_n\in C(\mathbb R_{+},\mathbb C)$
be Lagrange stable and $D(t)=D(a_1(t),\ldots,a_n(t))$ be the
discriminant of the equation (\ref{eqAE0}). Then the following
statements are equivalent:
\begin{enumerate}
\item $D(t)$ is separated from zero on $\mathbb R_{+} $;
\item there exists a positive number $\alpha$ such that
\begin{equation}\label{eqD1.1}
|D(b_1(t),\ldots,b_{n}(t))| \ge \alpha
\end{equation}
for every $(b_1,\ldots,b_{n})\in \omega_{(a_1,\ldots,a_{n})}$ and
$t\in \mathbb R_{+}$.
\end{enumerate}
\end{lemma}
\begin{proof} According to Lemma \ref{lD1} the relation
(\ref{eqD1}) implies (\ref{eqD1.1}).

We will show that (\ref{eqD1.1}) implies (\ref{eqD1}). If we
suppose that it is false, then there exists a sequence $t_k\to
+\infty$ as $k\to \infty$ such that
\begin{equation}\label{eqD1.2}
|D(t_k)|\to 0
\end{equation}
as $k\to \infty$. Consider the sequence
$a^{t_k}:=(a_{1}^{t_k},\ldots,a_{n}^{t_k})$ of functions from
$C(\mathbb R_{+},\mathbb C^{n})$. Since the functions $a_{i}\in
C(\mathbb R_{+},\mathbb C)$ ($i=1,\ldots,n$) are Lagrange stable
then we may suppose that the sequence $\{a^{t_k}\}$ converges in
$C(\mathbb R_{+},\mathbb C^{n})$. Denote by
$(b_{1},\ldots,b_{n})=\widetilde{a}:=\lim\limits_{k\to
\infty}a^{t_{k}}$ and
$\widetilde{D}(t):=D(b_{1}(t),\ldots,b_{n}(t))$. By (\ref{eqD1.1})
we have
\begin{equation}\label{eqD1.3}
|\widetilde{D}(t)|\ge \alpha
\end{equation}
for all $t\ge 0$.

On the other hand we have
\begin{equation}\label{eqD1.4}
|\widetilde{D}(0)|=\lim\limits_{k\to
\infty}|D(a_1^{t_k}(0),\ldots,a_{n}^{t_k}(0))|=\lim\limits_{k\to
\infty}|D(t_k)|.
\end{equation}
Taking into consideration (\ref{eqD1.2}) and (\ref{eqD1.4}) we
obtain
\begin{equation}\label{eqD1.5}
|\widetilde{D}(0)|=0.
\end{equation}
The relation (\ref{eqD1.5}) contradicts to (\ref{eqD1.3}). The
obtained contradiction proves our statement.
\end{proof}

Let $X$ be a compact metric space. A dynamical system $(X,\mathbb
T,\pi)$ is said to be point transitive if there exists a point
$x_0\in X$ such that $H(x_0)=X,$ where
$H(x_0):=\overline{\{\pi(t,x_0)|\ t\in \mathbb T\}}$.

An ambit \cite[Ch.VIII]{Aus_1988} (or $\mathbb T$ ambit) is a pair
$\langle (X,\mathbb T,\pi),x_0\rangle$, where $(X,\mathbb T,\pi)$
is a point transitive dynamical system $(X,\mathbb T,\pi)$ and
$x_0\in X$ with $H(x_0)=X$.

If $\langle (X,\mathbb T,\pi),x_0\rangle$ and $\langle (Y,\mathbb
T,\sigma),y_0\rangle$ are ambits and $h$ is a homomorphism of the
dynamical system $(X,\mathbb T,\pi)$ into $(Y,\mathbb T,\sigma)$
with $h(x_0)=y_0$ (it is clear that the homomorphism $h$ is
unique). In this case we write $\langle (X,\mathbb
T,\pi),x_0\rangle \ge \langle (Y,\mathbb T,\sigma),y_0\rangle$.

A universal $\mathbb T$ ambit is a $\mathbb T$ ambit $\langle
(U,\mathbb T,\lambda),u_0\rangle$ (with the compact phase space
$U$) such that $\langle (U,\mathbb T,\lambda),u_0\rangle \ge
\langle (X,\mathbb T,\pi),x_0\rangle$ for every $\mathbb T$ ambit
$\langle (X,\mathbb T,\pi),x_0\rangle$.

\begin{theorem}\label{thAU1} \cite[Ch.VIII]{Aus_1988} There exists a unique (up to a
homeomorphism) universal $\mathbb T$ ambit.
\end{theorem}

\begin{remark}\label{remAU1} \cite[Ch.VIII]{Aus_1988} Note that the phase space $U$ of the
universal $\mathbb T$ ambit $\langle (U,\mathbb
T,\lambda),u_0\rangle$ is homeomorphic to the Stone-Cech
compactification $\beta \mathbb T$ of $\mathbb T$.
\end{remark}

\begin{theorem}\label{thBC1}\cite[Remark 5 (item (i))]{RC_1967} Let $A_{i}\in C(\beta\mathbb R,\mathbb
C)$ ($i,\ldots,n$) then the equation
$$
x^{n}+A_{1}(u)x^{n-1}+\ldots +A_{n-1}(u)x+A_{n}(u)=0
$$
has at least one solution $\lambda \in C(\beta\mathbb R,\mathbb
C)$.
\end{theorem}

\begin{definition}\label{defCD} A pair $\langle (Y,\mathbb R,\sigma),\varphi \rangle$ is a
compactification of $(X,\mathbb R,\pi)$ if $\varphi$ is a
homomorphism $(X,\mathbb R,\pi)$ into $(Y,\mathbb R,\sigma)$ such
that $\overline{\varphi(X)}=Y$. A compactification $\langle
(Y,\mathbb R,\sigma),\varphi \rangle$ of $(X,\mathbb R,\sigma)$ is
called maximal (universal) if for any other compactification
$\langle (Z,\mathbb R,\delta),\psi \rangle$ of $(X,\mathbb R,\pi)$
there exists a homomorphism $\phi :(Z,\mathbb R,\delta)\to
(Y,\mathbb R,\sigma)$ such that $\overline{\phi(Z)}=Y$ and
$\varphi =\psi\dot \phi$.
\end{definition}

The dynamical system $(X,\mathbb R,\pi)$ is called a minimal and
almost periodic if there exists an almost periodic point $x_0\in
X$ such that $X=H(x_0)$.

\begin{theorem}\label{thBU1} \cite{Chu_1962} Let $(X,\mathbb R,\pi)$ be a minimal almost periodic dynamical system.
Then there is a universal almost periodic minimal dynamical system
$(U,\mathbb R,\sigma)$. If $(Y_{1},\mathbb R,\delta_1)$ and
$(Y_{2},\mathbb R,\delta_{2})$ are any two universal almost
periodic minimal dynamical systems, then $(Y_{1},\mathbb
R,\delta_{1})$ and $(Y_{2},\mathbb R,\delta_{2})$ are
homeomorphic.
\end{theorem}

Recall \cite{Shu_1978} that the Bohr compactification of a
topological group $G$ is a pair $(G_{B},i_{B})$, where $G_{B}$ is
a compact group and $i_{B}$ a homomorphism $G\to G_{B}$, such that
for any homomorphism $\varphi: G\to \Gamma$, where $\Gamma$ is a
compact group, there is a unique homomorphism
$\varphi_{B}:G_{B}\to \Gamma$ such that $\varphi =\varphi_{B}\circ
i_{B}$. The pair $(G,i_{G})$ is uniquely determined by this
property up to isomorphism and $\overline{i_{B}(G)}=G_{B}$.

\begin{remark}\label{remBU1} \cite[Lemma 9]{Chu_1962}, \cite[p.4]{Shu_1978} Let $(X,\mathbb
R,\pi)$ be a minimal almost periodic dynamical system and
$(U,\mathbb R,\sigma)$ be a universal almost periodic minimal
dynamical system for $(X,\mathbb R,\pi)$, then $U$ is homeomorphic
to Bohr compactification $b\mathbb R$ of $ \mathbb R$.
\end{remark}

 Everywhere below we suppose that $U=b\mathbb R$ and denote by $(U,\mathbb
R,\mu)$ the universal almost periodic minimal dynamical system
defined as follow $\mu(t,p)=i_{B}(t)p$ (see, for example,
\cite[p.324]{Chu_1962}) for all $t\in \mathbb R$ and $p\in
U=b\mathbb R$, where $(b\mathbb R,i_{B})$ is the Bohr
compactification of $\mathbb R$.

Consider the following equation
\begin{equation}\label{eqAE3.3}
x ^{n}+A_{1}(u)x ^{n-1}+\ldots +A_{n-1}(u)x+A_{n}(u)=0\ \ (u\in
U),
\end{equation}
where $A_{i}\in C(U,\mathbb C)$ ($i=1,\ldots,n$) and
$(A_1,\ldots,A_{n}):=\Phi \in C(U,\mathbb C^{n})$.

\begin{lemma}\label{lGL1} If
\begin{equation}\label{eqB1}
D(u):=D(A_{1}(u),\ldots,A_{n}(u))\not= 0\nonumber
\end{equation}
for every $u\in U$, then the equation (\ref{eqAE3.3}) has exactly
$n$ different solutions $\nu_{i}\in C(U,\mathbb C)$
($i=1,\ldots,n$), i.e., for all $u\in U$ and $i\not= j$ ($i,j\in
\{1,\ldots,n\}$
\begin{equation}\label{eqB1_0}
\nu_{i}(u)\not= \nu_{j}(u).
\end{equation}
\end{lemma}
\begin{proof} This statement directly follows from \cite[Remark 1 to Theorem 1.4]{GL_1969} and
Remark \ref{remBU1}.
\end{proof}

\begin{lemma}\label{lLG1_01} Under the conditions of Lemma
\ref{lGL1} there exists a positive number $\alpha$ such that
\begin{equation}\label{eqE3_01}
|\nu_{i}(u)-\nu_{j}(u)|\ge \alpha \nonumber
\end{equation}
for all $u\in U$ and $i,j\in \{1,\ldots,n\}$ with $i\not= j$.
\end{lemma}
\begin{proof} If we assume that this statement is not true, then
there are sequences $\{u_k\}\subset U$, $\delta_{k}\to 0$
($\delta_{k}>0$) and $\{i_{k}\},\{j_{k}\}\subset \{1,\ldots,n\}$
such that
\begin{equation}\label{eqE3_02}
|\nu_{i_k}(u_k)-\nu_{j_{k}}(u_k)|\le \delta_{k}
\end{equation}
for every $k\in \mathbb N$. Taking into account that the space $U$
is compact and $i_k,j_k\in \{1,\ldots,n\}$ (for all $k\in \mathbb
N$) without loss of generality we may assume that the sequence
$\{u_k\}$ converges and there are $i_0,j_0\in \{1,\ldots,n\}$ with
$i_0\not= j_0$ such that $i_k=i_0$ and $j_k=j_0$ for every $k\in
\mathbb N$. Denote by
\begin{equation}\label{eqE3_03}
\bar{u}:=\lim\limits_{k\to \infty} u_{k}\nonumber
\end{equation}
then passing to the limit in (\ref{eqE3_02}) as $k\to \infty$ we
obtain
\begin{equation}\label{eqE3_04}
\nu_{i_0}(\bar{u})=\nu_{j_0}(\bar{u}).
\end{equation}
The relations (\ref{eqE3_04}) and (\ref{eqB1_0}) are
contradictory. The obtained contradiction proves our statement.
Lemma is proved.
\end{proof}

\begin{coro}\label{corI1} Under the conditions of Lemma \ref{lGL1}
there exists a positive constant $\alpha$ such that
\begin{equation}\label{eqI1}
d:=\inf\limits_{u\in U}\min\limits_{i\not=
j}|\nu_{i}(u)-\nu_{j}(u)|\ge \alpha .\nonumber
\end{equation}
\end{coro}

Along with the equation (\ref{eqAE3.3}) we consider a family of
equations
\begin{equation}\label{eqE3_1u}
x^{n}+A_{1}(\mu(t,u))x^{n-1}+\ldots
+A_{n-1}(\mu(t,u))x+A_{n}(\mu(t,u))=0 \ \ (u\in U).
\end{equation}

\begin{lemma}\label{lLG1_1} Let $\nu \in C(U,\mathbb C)$ be a
solution of the equation (\ref{eqAE3.3}).

The function $\lambda :\mathbb T \to \mathbb C$ defined by
\begin{equation}\label{eqE3_2}
\lambda(t):=\nu(\mu(t,u))\ \ (\forall \ t\in \mathbb T) \nonumber
\end{equation}
possesses the following properties:
\begin{enumerate}
\item $\lambda$ is a continuous solution of the equation
(\ref{eqE3_1u}); \item the solution $\lambda$ of the equation
(\ref{eqE3_1u}) is Lagrange stable.
\end{enumerate}
\end{lemma}
\begin{proof} The first statement is evident

To prove the second statement according to Lemma \ref{lA1.0} we
need to show that the function $\lambda$ is bounded and uniformly
continuous on $\mathbb T$.

Note that the function $\lambda \in C(\mathbb T,\mathbb C)$
($i=1,\ldots,n$) is bounded on $\mathbb T$. Indeed, denote by
\begin{equation}\label{eqAE4.1}
A:=\sup\limits_{1\le i\le n}\max\limits_{u\in U}|A_{i}(u)|
\nonumber
\end{equation}
then by Lemma \ref{lAE1} we have
\begin{equation}\label{eqAE6_1}
|\lambda(t)|\le 1+C \nonumber
\end{equation}
for every $t\in \mathbb T$.

Now we will show that $\lambda$ is uniformly continuous on
$\mathbb T$. If we assume that it is false, then there are
$\varepsilon_{0}>0$, $\delta_{k}\to 0$ ($\delta_{k}>0$) and
sequences $\{t_k\}\subset \mathbb T$ and $\{h_k\}$ with
$|h_k|<\delta_{k}$ such that
\begin{equation}\label{eqAE6.1}
|\lambda(t_k+h_k)-\lambda(t_k)|\ge \varepsilon_{0} \nonumber
\end{equation}
for every $k\in \mathbb N$.

Since the space $U$ is compact and $\{\mu(t_k,u)\}\subset U$ then
without loss of generality we can assume that the sequence
$\{\mu(t_k,u)\}$ converges. Denote its limit by
\begin{equation}\label{eqAE6.2}
\bar{u}=\lim\limits_{k\to \infty}\mu(t_k,u)
\end{equation}
and note that
\begin{equation}\label{eqAE6.3}
\lim\limits_{k\to \infty}\mu(t_k+h_k,u)=\lim\limits_{k\to
\infty}\mu(h_k,\mu(t_k,u))=\bar{u} .
\end{equation}

On the other hand we have
$$
\varepsilon_0\le
|\lambda(t_k+h_k)-\lambda(t_k)|=|\nu(\mu(t_k+h_k))-\nu(\mu(t_k,u))|=
$$
\begin{equation}\label{eqAE6.4}
|\nu(\mu(h_k,\mu(t_k,u)))-\nu(\mu(t_k,u))|
\end{equation}
for all $k\in \mathbb N$. Passing to the limit in (\ref{eqAE6.4})
as $k\to \infty$ and taking into account (\ref{eqAE6.2}) and
(\ref{eqAE6.3}) we obtain $\varepsilon_0\le 0$. The last
inequality contradicts to the choice of the number
$\varepsilon_0$. The obtained contradiction proves our statement.
Lemma is completely proved.
\end{proof}

\begin{remark}\label{remBC1} Note that Lemma \ref{lLG1_1} remains
true if we replace $U (=b\mathbb R)$ by $\beta \mathbb R$.
\end{remark}

\begin{lemma}\label{lBC1.1} Let $a_{i}\in C(\mathbb R_{+},\mathbb
C)$ ($i=1,\ldots,n$) be positively Lagrange stable. Then the
equation (\ref{eqAE0}) has at least one positively Lagrange stable
solution.
\end{lemma}
\begin{proof}

To show this fact we will consider the equation
\begin{equation}\label{eqAE3a}
x^{n}+\widetilde{a}_1(t)x^{n-1}+\ldots
+\widetilde{a}_{n-1}(t)x+\widetilde{a}_{n}(t)=0,
\end{equation}
where $\widetilde{a}_{i}(t)=a_{i}(t)$ for all $t\ge 0$ and
$\widetilde{a}_{i}(t)=a_{i}(0)$ for every $t<0$ ($i=1,\ldots,n$).
Denote by $widetilde{F}$ the mapping from $H(\widetilde{a})$
($\widetilde{a}:=(\widetilde{a}_{1},\ldots,\widetilde{a}_{n})$)
into $\mathbb C^{n}$ defined by
$$
\widetilde{F}(\widetilde{b}):=\widetilde{b}(0)
$$
for every $\widetilde{b}\in H(\widetilde{a})$. By the universality
of the dynamical system $(\beta \mathbb R,\mathbb R,\mu)$ there
exists a homomorphism $h:(\beta \mathbb R,\mathbb R,\mu)\to
(H(\widetilde{a},\mathbb R,\sigma)$. Let $\widetilde{\Phi}$ be the
mapping from $\beta \mathbb R$ into $\mathbb C^{n}$ defined by
$$
\widetilde{\Phi}(u):=\widetilde{F}(h(u))
$$
for all $u\in \beta \mathbb R$. Note that
\begin{equation}\label{eqDC}
\widetilde{\Phi}(u_0)=\widetilde{a}(0)\ \ \mbox{and}\ \
\widetilde{\Phi}(\mu(t,u_0))=\widetilde{a}(t)\ \ (\forall\ t\in
\mathbb R).\nonumber
\end{equation}

By Lemma \ref{lLG1_1} the equation (\ref{eqAE0}) admits at least
one Lagrange stable solution $\widetilde{\lambda}\in C(\mathbb
R,\mathbb C)$. To finish the proof of Lemma it suffices to note
that the function $\lambda \in C(\mathbb R_{+},\mathbb C)$,
defined by $\lambda(t):=\widetilde{\lambda}(t)$ for all $t\in
\mathbb R_{+}$, is a positively Lagrange stable solution of the
equation (\ref{eqAE0}).
\end{proof}

\begin{lemma}\label{lLG1_2} Under the conditions of Lemma
\ref{lGL1} there exists a positive number $\alpha$ such that
\begin{equation}\label{eqAE_8}
|\nu_{i}(\mu(t,u))-\nu_{j}(\mu(t,u))|\ge \alpha >0 \nonumber
\end{equation}
for all $t\in \mathbb R$ and $i,j=1,\ldots,n$ with $i\not= j$.
\end{lemma}
\begin{proof} This statement follows from Lemmas \ref{lGL1} (see also Corollary \ref{corI1}) and \ref{lLG1_1}. Indeed, we
have
\begin{equation}\label{eqAE_9}
|\nu_{i}(\mu(t,u))-\nu_{j}(\mu(t,u))|\ge \min\limits_{i\not=
j}|\nu_{i}(\mu(t,u))-\nu_{j}(\mu(t,u))| \ge \nonumber
\end{equation}
$$
\inf\limits_{u\in U}\min\limits_{i\not=
j}|\nu_{i}(u)-\nu_{j}(u)|\ge \alpha
$$
for all $t\in \mathbb T$ and $i,j\in \{1,\ldots,n\}$ with $i\not=
j$.
\end{proof}

Let $a=(a_1,\ldots,a_n)\in C(\mathbb R,\mathbb C^{n})$ be an
almost periodic function, $H(a)=\overline{\{a^{h}|\ h\in \mathbb
R\}}$ and $(H(a),\mathbb R,\sigma)$ be the shift dynamical system
on $H(a)$. By universality of the dynamical system $(b\mathbb
R,\mathbb R,\mu)$ there exists a homomorphism $h:(b\mathbb
R,\mathbb R,\mu)\to (H(a),\mathbb R,\sigma)$. Denote by $F$ the
mapping from $H(f)$ into $\mathbb C^{n}$ defined by the equality
\begin{equation}\label{eqAE3.2}
F(g):=g(0) \nonumber
\end{equation}
for every $g\in H(f)$. Let $\Phi$ be the mapping from $U$ into
$\mathbb C^{n}$ defined by
\begin{equation}\label{eqAE3.03}
\Phi(u):=F(h(u)) \nonumber
\end{equation}
for all $u\in U$. It is clear that
\begin{equation}\label{eqAE3.4}
\Phi(u_0)=a(0)\ \ \mbox{and}\ \
\Phi(\mu(t,u_0))=a(t)=(a_1(t),\ldots,a_{n}(t))\ \ (\forall \ \
t\in \mathbb T). \nonumber
\end{equation}

\begin{lemma}\label{lGL1.1} Let $a_{i}\in C(\mathbb R,\mathbb C)$
($i=1,\ldots,n$) be almost periodic. The following statements are
equivalent:
\begin{enumerate}
\item[a.]
\begin{equation}\label{eqB1.1}
D(u):=D(A_{1}(u),\ldots,A_{n}(u))\not= 0
\end{equation}
for all $u\in U$; \item[b.] there exists a positive number
$\alpha$ such that
\begin{equation}\label{eqB1.2}
|D(t)|=|D(a_1(t),\ldots,a_n(t))|\ge \alpha
\end{equation}
for every $t\in \mathbb T$.
\end{enumerate}
\end{lemma}
\begin{proof}
Assume that $A_{i}\in C(U,\mathbb C)$ ($i=1,\ldots,n$) and the
condition (\ref{eqB1.1}) holds. Since the function $D:U\to \mathbb
C$ is continuous (as the composition of continuous functions) and
the space $U$ ($=b \mathbb T$) is compact then there exists an
element $\bar{u}\in U$ such that
\begin{equation}\label{eqB1.3}
\alpha :=\min\limits_{u\in U}|D(u)|=\min\limits_{u\in
U}|D(A_{1}(u),\ldots,A_{n}(u))|=|D(\bar{u})|>0. \nonumber
\end{equation}
Note that
\begin{equation}\label{eqB1.4}
|D(t)|=|D(a(t))|=|D(a_1(t),\ldots,a_{n}(t))|=|D(\Phi(\mu(t,u_0))|\ge
\nonumber
\end{equation}
$$
\min\limits_{u\in U}|D(u)|=\alpha >0
$$
for all $t\in \mathbb T$.

Assume that the condition (\ref{eqB1.2}) holds. Let $u\in U$ be an
arbitrary element from $U$. Since $U=H(u_0)$ then there exists a
sequence $\{t_k\}\subset \mathbb T$ such that $\mu(t_k,u_0)\to u$
as $k\to \infty$ and, consequently,
\begin{equation}\label{eqB1.5}
|D(u)|=\lim\limits_{k\to
\infty}|D(\mu(t_k,u_0))|=\lim\limits_{k\to
\infty}|D(a_{1}^{t_k},\ldots,a_{n}^{t_k})|\ge \inf\limits_{t\in
\mathbb T}|D(t+t_k)|\ge \alpha .\nonumber
\end{equation}
\end{proof}

\begin{lemma}\label{lGL2} Let $a_{i}\in C(\mathbb R,\mathbb C)$
($i=1,\ldots,n$) be almost periodic. Assume that the discriminant
$D(t):=D(a_1(t),\ldots,a_n(t))$ of the equation (\ref{eqAE0})
satisfies the condition
\begin{equation}\label{eqGL1}
\inf\limits_{t\in \mathbb R}|D(t)|>0. \nonumber
\end{equation}

Then the following statements hold:
\begin{enumerate}
\item the equation (\ref{eqAE0}) has exactly $n$ Lagrange stable
solutions $\lambda_{i}\in C(\mathbb R,\mathbb C)$
($i=1,\ldots,n$); \item there exists a positive number $\alpha$
such that
\begin{equation}\label{eqGL2}
|\lambda_{i}(t)-\lambda_{j}(t)|\ge \alpha \nonumber
\end{equation}
for every $i\not= j$ ($i,j=1,\ldots,n$) and $t\in \mathbb R$.
\end{enumerate}
\end{lemma}
\begin{proof}
Let $\nu_{1}(u),\ldots,\nu_{n}(u)$ be all solutions (counting
multiplicities) of the equation (\ref{eqAE3.3}). Note that
$a(t)=(a_{1}(t),\ldots,a_{n}(t))=\Phi(\mu(t,u_0))$. If $\lambda
\in C(\mathbb R,\mathbb C^{n})$ is a continuous solution of the
equation (\ref{eqAE0}), then
\begin{equation}\label{eqDI0}
\lambda(t)\in \{\nu_{1}(\mu(t,u_0)),\ldots,\nu_{n}(\mu(t,u_0))\}
\nonumber
\end{equation}
for all $t\in \mathbb R$ and, consequently, there exists a number
$i_0\in \{1,\ldots,n\}$ such that
\begin{equation}\label{eqGL3}
\lambda(t)=\nu_{i_0}(\mu(t,u_0))
\end{equation}
for all $t\in \mathbb R$. Indeed. If we assume that it is not
true, then there exist $i_0,j_0\in \{1,\ldots,n\}$ ($i_0\not=
j_0$) and $t_{1},t_{2}\in \mathbb R$ ($t_{1}<t_{2}$) such that
\begin{equation}\label{eqDI1}
\lambda(t_1)=\nu_{i_0}(\mu(t_1,u_0))\ \ \mbox{and}\ \
\lambda(t_2)=\nu_{j_0}(\mu(t_2,u_0)).\nonumber
\end{equation}
Consider the function $\beta\in C(\mathbb R,\mathbb R_{+})$
defined by
\begin{equation}\label{eqDI2}
\beta(t):=|\lambda(t)-\nu_{i_0}(\mu(t,u_0))| \nonumber
\end{equation}
for all $t\in \mathbb R$. Note that
\begin{eqnarray}\label{eqDI3}
& \beta(t_1)=0\ \ \mbox{and}\ \ \
\beta(t_2)=|\lambda(t_2)-\nu_{i_0}(\mu(t_2,u_0))|=\\
& |\nu_{j_0}(\mu(t_2,u_0))-\nu_{i_0}(\mu(t_2,u_0))| \ge \alpha
.\nonumber
\end{eqnarray}
We fix $\ell \in \mathbb N$ such that
\begin{equation}\label{eqDI4}
\ell >1\ \ \mbox{and}\ \ \frac{d}{\ell}<\alpha . \nonumber
\end{equation}
Since the function $\beta$ is continuous then taking into account
(\ref{eqDI3}) we conclude that there exists a number $\bar{t}\in
(t_1,t_2)$ such that
\begin{equation}\label{eqDI5}
\beta(\bar{t})=|\lambda(\bar{t})-\nu_{i_0}(\mu(\bar{t},u_0))|=\frac{d}{\ell}.\nonumber
\end{equation}
According to (\ref{eqDI0}) there exists $i_{0}'\in \{1,\ldots,n\}$
($i_{0}'\not= i_{0}$) such that
\begin{equation}\label{eqDI_6}
\lambda(\bar{t})=\nu_{i_{0}'}(\mu(\bar{t},u_0)).\nonumber
\end{equation}
On the other hand
\begin{eqnarray}\label{eqDI6}
& d\le
|\nu_{i_0}(\mu(\bar{t},u_0))-\nu_{i_{0}'}(\mu(\bar{t},u_0))|\le \\
& |\nu_{i_0}(\mu(\bar{t},u_0))-\lambda(\bar{t})| +
|\lambda(\bar{t})-\nu_{i_{0}'}(\mu(\bar{t},u_0))|\nonumber
\end{eqnarray}
and taking into account (\ref{eqDI3}) from (\ref{eqDI6}) we obtain
\begin{equation}\label{eqDI7}
d\le \frac{d}{\ell} \nonumber
\end{equation}
and, consequently, $\ell \le 1$. The last relation contradicts the
choice of the number $\ell$. The obtained contradiction proves our
statement.

By Lemma \ref{lLG1_1} and (\ref{eqGL3}) follows that the function
$\lambda :\mathbb R\to \mathbb C$ is uniformly continuous on
$\mathbb R$.

The second statement of Lemma follows from Lemmas \ref{lGL1},
\ref{lLG1_2}, \ref{lGL1.1} and the first statement.
\end{proof}

\begin{theorem}\label{thAE1} Assume that the following conditions
are fulfilled:
\begin{enumerate}
\item the functions $a_{i}\in C(\mathbb R_{+},\mathbb C)$ are
remotely almost periodic and positively Lagrange stable; \item the
$\omega$-limit set $\omega_{(a_1,\ldots,a_{n})}$ of the function
$a=(a_1,\ldots,a_{n})\in C(\mathbb R_{+},\mathbb C^{n})$ is
minimal; \item there exists a positive number $\alpha$ such that
\begin{equation}\label{eqDD_001}
|D(t)|=|D(a_1(t),\ldots,a_{n}(t))|\ge \alpha
\end{equation}
for all $t\ge 0$.
\end{enumerate}

Then every continuous and positively Lagrange stable solution
$\lambda :\mathbb R_{+}\to \mathbb C$ of the equation
(\ref{eqAE0}) is remotely almost periodic.
\end{theorem}
\begin{proof}
Let $a\in C(\mathbb T,\mathbb C^{n})$ and $Z:=H(a)$. Denote by
$(H(a),\mathbb T,\sigma)$ (respectively, by $(\Omega_{a},\mathbb
T,\sigma)$) the shift dynamical system on $H(a)$ (respectively, on
$\Omega_{a}$).

Consider a continuous solution $\lambda :\mathbb T\to \mathbb C$
of the equation (\ref{eqAE0}). Denote by
$X:=H(\lambda,a):=\overline{\{(\lambda^{h},a^{h}):\ h\in \mathbb
T\}}$ (respectively, $X:=\Omega_{(\lambda,a)}$) and $(X,\mathbb
T,\pi)$ be the shift dynamical system on the set $X$ induced by
the product dynamical system $(C(\mathbb T,\mathbb C)\times
C(\mathbb T,\mathbb C^{n}),\mathbb T,\sigma)$. Finally we put
$h:=pr_{2}:X \to Z$ then the triplet
\begin{equation}\label{eqNDS_1}
\langle (X,\mathbb T,\pi),(Z,\mathbb T,\sigma),h\rangle
\end{equation}
is a nonautonomous dynamical system associated by the algebraic
equation (\ref{eqAE0}) \cite{Bro_73}, \cite{Zhi69}.

Denote by $Y:=\omega_{a}=\omega_{(a_1,\ldots,a_n)}$. Since
$\omega_{a}$ is a nonempty, compact and minimal (and,
consequently, invariant) subset of the dynamical system
$(Z,\mathbb R,\sigma)$ then on $Y$ is induced the shift dynamical
system $(Y,\mathbb R,\sigma)$.

According to Lemma \ref{lD1.1} the condition (\ref{eqDD_001}) is
equivalent to  the following:
\begin{equation}\label{eqNDS_21}
|D(b_{1},\ldots,b_{n})|\ge \alpha \nonumber
\end{equation}
for every $(b_{1},\ldots,b_{n})\in \omega_{(a_1,\ldots,a_n)}$.

By Lemma \ref{lLG1_2} every point
$(\widetilde{\lambda},\widetilde{a})\in \omega_{(\lambda,a)}$ of
the nonautonomous dynamical system (\ref{eqNDS_1}) is separated in
the set $\omega_{(\lambda,a)}$. Now to finish the proof of Theorem
it suffices to apply Theorem \ref{thNDS2}. Theorem is proved.
\end{proof}

\begin{coro}\label{corAE01} Assume that the following conditions
are fulfilled:
\begin{enumerate}
\item the functions $a_{i}\in C(\mathbb R_{+},\mathbb C)$ are
asymptotically almost periodic;
\item there exists a positive
number $\alpha$ such that
\begin{equation}\label{eqDD01}
|D(t)|=|D(a_1(t),\ldots,a_{n}(t))|\ge \alpha \nonumber
\end{equation}
for every $t\ge 0$.
\end{enumerate}

Then every continuous and positively Lagrange stable solution
$\lambda :\mathbb R_{+}\to \mathbb C$ of the equation
(\ref{eqAE0}) is remotely almost periodic.
\end{coro}
\begin{proof}
Assume that the functions $a_{i}\in C(\mathbb R_{+},\mathbb C)$
are asymptotically almost periodic then
\begin{enumerate}
\item they are remotely almost periodic; \item the function
$(a_1,\ldots,a_{n})\in C(\mathbb R_{+},\mathbb C^{n})$ is also
asymptotically almost periodic; \item the $\omega$-limit set
$\omega_{(a_1,\ldots,a_{n})}$ of the function $(a_1,\ldots,a_{n})$
is minimal and almost periodic.
\end{enumerate}

Now to finish the proof of Corollary \ref{corAE01} it suffices to
apply Theorem \ref{thAE1}.
\end{proof}

The following natural question arises.

\textbf{Question.} Are the solutions of the equation (\ref{eqAE0})
(under the conditions of Corollary \ref{corAE01}) asymptotically
almost periodic.

The following theorem gives a positive answer to this question.

\begin{theorem}\label{thAE2} Assume that the following conditions
are fulfilled:
\begin{enumerate}
\item the functions $a_{i}\in C(\mathbb R_{+},\mathbb C)$ are
asymptotically almost periodic;
\item there exists a positive
number $\alpha$ such that
\begin{equation}\label{eqDD_01}
|D(t)|=|D(a_1(t),\ldots,a_{n}(t))|\ge \alpha
\end{equation}
for all $t\ge 0$.
\end{enumerate}

Then every continuous and positively Lagrange stable solution
$\lambda :\mathbb R_{+}\to \mathbb C$ of the equation
(\ref{eqAE0}) is asymptotically almost periodic.
\end{theorem}
\begin{proof}
This statement can be proved using the same arguments as in the
proof of Theorem \ref{thAE1} (see also the proof of Lemma
\ref{lAE1}) but unless of Theorem \ref{thNDS2} we need to apply
Theorem \ref{t2.3.6}.
\end{proof}

\begin{remark}\label{remRAPes} Note that the condition
(\ref{eqDD_001}) in Theorem \ref{thAE1} is essential.
\end{remark}

Below we will give an example which confirm this statement.

\begin{example}\label{exRAPes1}
{\rm Zhikov \cite[p.122]{Zhi69} constructed an example of an
almost periodic function $f\in C(\mathbb R,\mathbb R)$ satisfying
the conditions
\begin{equation}\label{eqVV1}
|f(t)|>0\ (\forall \ \ t\in \mathbb R), \ \ \inf\limits_{t\in
\mathbb R}|f(t)|=0
\end{equation}
and such that the solutions of the equation $X^2-f(t)=0$ are not
almost periodic.

Consider the function $p\in C(\mathbb R,\mathbb R)$ defined by the
equality
\begin{equation}\label{eqVV2}
p(t):=f(t)+e^{-t} \nonumber
\end{equation}
for all $t\in \mathbb R$, where $f$ is the almost periodic
function from Zhikov's example \cite{Zhi69}.

\begin{lemma}\label{lVV1} Function $p$ possesses the following properties:
\begin{enumerate}
\item the function $p$ is asymptotically almost periodic and,
consequently, it is remotely almost periodic; \item
$\omega_{p}=H(f)$ and, consequently, the set $\omega_{p}$ is
minimal;\item
\begin{equation}\label{eqVV3}
\inf\limits_{t\in \mathbb R_{+}}|p(t)|=0 .\nonumber
\end{equation}
\end{enumerate}
\end{lemma}
\begin{proof} The first two statements of Lemma are evident.

Now we will show the third statement. If we assume that it is not
true, then there exists a positive number $\alpha$ such that
\begin{equation}\label{eqVV4}
|p(t)|\ge \alpha \nonumber
\end{equation}
or equivalently
\begin{equation}\label{eqVV4.1}
|f(t)+e^{-t}|\ge \alpha
\end{equation}
 for all $t\in \mathbb R_{+}$. Since the function $f$ is almost
periodic then there exists a sequence $\{t_k\}\subset \mathbb R$
such that $t_k\to +\infty$ and $f(t+t_k)\to f(t)$ (uniformly
w.r.t. $t\in \mathbb R$) as $k\to \infty$. Let now $s$ be an
arbitrary number from $\mathbb R$ then there exists a number
$k_0=k(s)\in \mathbb N$ such that
\begin{equation}\label{eqVV5}
s+t_k\ge 0
\end{equation}
for all $k\ge k_0$. From (\ref{eqVV4.1}) and (\ref{eqVV5}) we
receive
\begin{equation}\label{eqVV6}
|f(s+t_k)+e^{-s-t_k}|\ge \alpha
\end{equation}
for every $k\ge k_0$. Passing to the limit in (\ref{eqVV6}) as
$k\to \infty$ we obtain
\begin{equation}\label{eqVV7}
|f(s)|\ge \alpha
\end{equation}
for all $s\in \mathbb R$. The relations (\ref{eqVV1}) and
(\ref{eqVV7}) are contradictory. The obtained contradiction proves
our statement. Lemma is proved.
\end{proof}

Consider the algebraic equation
\begin{equation}\label{eqVV8}
\lambda^{2}-p(t)=0.
\end{equation}
It is easy to check that the condition (\ref{eqDD_01}) is not
fulfilled.

\begin{lemma}\label{lVV2} The equation (\ref{eqVV8}) has not
positively Lagrange stable remotely almost periodic solutions.
\end{lemma}
\begin{proof} Assume that this statement is false, i.e., there
exist a positively Lagrange stable and remotely almost periodic
solution $\widetilde{\lambda}$. By Lemma \ref{lRAP1} the
$\omega$-limits set $\omega_{\widetilde{\lambda}}$ of
$\widetilde{\lambda}$ is a nonempty, compact, invariant and
equi-almost periodic set. In particular, the set
$\omega_{\widetilde{\lambda}}$ consists of almost periodic
functions. Let $\{t_k\}$ be a sequence as in the proof of Lemma
\ref{lVV1}. Consider the sequence $\widetilde{\lambda}^{t_k}$.
Since the function $\widetilde{\lambda}$ is positively Lagrange
stable then without loss of generality we can suppose that the
sequence $\widetilde{\lambda}^{t_k}$ converges. Denote its limit
by $\lambda$. It is clear that $\lambda \in
\omega_{\widetilde{\lambda}}$ and, consequently, the function
$\lambda$ is almost periodic.

By construction of the sequence $\{t_k\}$ we have
\begin{equation}\label{eqVV9}
f^{t_k}\to f,\ p^{t_k}\to f \ \mbox{and}\ \
\widetilde{\lambda}^{t_k}\to \lambda \nonumber
\end{equation}
as $k\to \infty$ (in the space $C(\mathbb R,\mathbb R)$) and,
consequently, $\lambda$ is an almost periodic solution of the
equation $X^{2}-f(t)=0$. The last statement contradicts to the
Zhikov's result \cite[p.122]{Zhi69}. The obtained contradiction
proves our statement. Lemma is proved.
\end{proof}
}
\end{example}

\textbf{Open problem.} The question, is whether Theorem
\ref{thAE1} remains true in the general case when the set
$\omega_{(a_1,\ldots,a_{n})}$ is not minimal, is open.

\section{Funding}

This research was supported by the State Program of the Republic
of Moldova "Remotely Almost Periodic Solutions of Differential
Equations (25.80012.5007.77SE)" and partially was supported by the
Institutional Research Program 011303 "SAT\-GED", Moldova State
University.

\section{Data availability}

No data was used for the research described in the article.

\section{Conflict of Interest}

The author declares that he does not have conflict of interest.

\section{ORCID ID}

\textbf{D. Cheban} https://orcid.org/0000-0002-2309-3823

\end{document}